\documentclass[12pt]{amsart}
\usepackage{amssymb, amsthm, amsmath, enumerate, graphicx, hyperref, mathtools, comment,extarrows,microtype}
\usepackage[all]{xy}
\usepackage{xcolor}

\newcommand{\C}{\mathbb{C}}
\newcommand{\Q}{\mathbb{Q}}
\newcommand{\R}{\mathbb{R}}
\renewcommand{\H}{\mathbb{H}}
\newcommand{\Z}{\mathbb{Z}}
\newcommand{\N}{\mathbb{N}}

\newcommand{\bbC}{\mathbb{C}}

\newcommand{\bbH}{\mathbb{H}}

\newcommand{\bbQ}{\mathbb{Q}}
\newcommand{\bbR}{\mathbb{R}}

\newcommand{\bbZ}{\mathbb{Z}}

\newcommand{\calE}{\mathcal{E}}
\newcommand{\calF}{\mathcal{F}}
\newcommand{\calG}{\mathcal{G}}
\newcommand{\calH}{\mathcal{H}}

\newcommand{\calK}{\mathcal{K}}

\newcommand{\cusps}[1]{\Omega{\parent{#1}}}
\newcommand{\cusp}[1]{#1}
\newcommand{\cuspwidth}[2]{w_{#2}}

\newcommand{\Znu}[1]{{\Z}^{\parent{#1}}}
\newcommand{\Zrho}[2]{{\Z}_{#2}^{\parent{#1}}}

\newcommand{\I}[2]{\widetilde I_{#2}^{\parent{#1}}} 
\newcommand{\tI}[2]{I_{#2}^{\parent{#1}}} 
\newcommand{\J}[2]{\widetilde J_{#2}^{\parent{#1}}} 
\newcommand{\tJ}[2]{J_{#2}^{\parent{#1}}} 

\newcommand{\spceis}[1]{\calE_{#1}} 

\newcommand{\tr}[2]{\operatorname{Tr}_{#1}^{#2}}

\newcommand{\spcfinf}[1]{M^{\parent{\infty}}_{#1}} 
\newcommand{\spcginf}[1]{\widehat{M}^{\parent{\infty}}_{#1}} 
\newcommand{\spcG}[1]{{S}_{#1}} 

\newcommand{\finfgen}[2]{\calH^{\parent{#1}}_{#2}}

\newcommand{\finf}[3]{f^{\parent{#1}}_{#2, #3}} 
\newcommand{\ainf}[4]{a^{\parent{#1}}_{#2}(#3,#4)} 

\newcommand{\ginf}[3]{g^{\parent{#1}}_{#2, #3}} 
\newcommand{\binf}[4]{b^{\parent{#1}}_{#2}(#3,#4)} 

\newcommand{\qinf}[3]{\mathbf q^{\parent{#1}}_{#2, #3}}
\newcommand{\qinfpart}[4]{ \mathrm q^{\parent{#1}, #4}_{#2, #3}}

\newcommand{\finferror}[3]{\mathbf \calF^{\parent{#1}}_{#2, #3}}
\newcommand{\finferrorpart}[4]{\mathbf \calF^{\parent{#1}, #4}_{#2, #3}}

\newcommand{\principalfinf}[3]{\mathbf f^{\parent{#1}}_{#2, #3}} 
\newcommand{\principalfinfpart}[4]{\parent{\mathbf f^{\parent{#1}}_{#2, #3}}^{#4}} 

\newcommand{\principalginf}[3]{\mathbf g^{\parent{#1}}_{#2, #3}} 
\newcommand{\principalginfpart}[4]{\parent{\mathbf g^{\parent{#1}}_{#2, #3}}^{#4}} 

\newcommand{\ginferror}[3]{\mathbf \calG^{\parent{#1}}_{#2, #3}}
\newcommand{\ginferrorpart}[4]{\mathbf \calG^{\parent{#1}, #4}_{#2, #3}}

\newcommand{\Finf}[3]{F^{\parent{#1}}_{#2, #3}} 
\newcommand{\Ainf}[4]{A^{\parent{#1}}_{#2}(#3,#4)} 

\newcommand{\Ginf}[3]{G^{\parent{#1}}_{#2, #3}} 
\newcommand{\Binf}[4]{B^{\parent{#1}}_{#2}(#3,#4)} 

\newcommand{\Fvan}[4]{F^{\parent{#1}}_{#2, #4, #3}} 
\newcommand{\Avan}[5]{A^{\parent{#1}}_{#2, #5}(#3,#4)} 

\newcommand{\Ivan}[3]{\widetilde I_{#2}^{\parent{#1}}(#3)} 
\newcommand{\tIvan}[3]{I_{#2}^{\parent{#1}}(#3)} 

\newcommand{\fvan}[4]{f^{\parent{#1}}_{#2, #4, #3}} 
\newcommand{\avan}[5]{a^{\parent{#1}}_{#2, #5}(#3,#4)} 

\newcommand{\Ialpha}[3]{\widetilde I_{#2}^{\parent{#1,#3}}} 
\newcommand{\tIalpha}[4]{I_{#2}^{\parent{#1,#3,#4}}} 
\newcommand{\Jalpha}[3]{\widetilde J_{#2}^{\parent{#1,#3}}} 
\newcommand{\tJalpha}[4]{J_{#2}^{\parent{#1, #3, #4}}} 

\newcommand{\qalpha}[4]{\mathbf q^{\parent{#1, #4}}_{#2, #3}} 
\newcommand{\qalphapart}[5]{\mathrm q^{\parent{#1, #4},#5}_{#2, #3}}

\newcommand{\falpha}[5]{f^{\parent{#1, #4, #5}}_{#2, #3}} 
\newcommand{\aalpha}[6]{a^{\parent{#1, #5, #6}}_{#2}\parent{#3, #4}} 

\newcommand{\principalfalpha}[5]{\mathbf f^{\parent{#1, #4, #5}}_{#2, #3}} 
\newcommand{\principalfalphapart}[6]{\mathbf f^{\parent{#1, #4, #5},#6}_{#2, #3}} 

\newcommand{\falphaerror}[5]{\mathbf \calF^{\parent{#1, #4, #5}}_{#2, #3}}
\newcommand{\falphaerrorpart}[6]{ \calF^{\parent{#1, #4, #5},#6}_{#2, #3}}

\newcommand{\galpha}[5]{g^{\parent{#1, #4, #5}}_{#2, #3}} 
\newcommand{\balpha}[6]{b^{\parent{#1, #5, #6}}_{#2}\parent{#3, #4}} 

\newcommand{\principalgalpha}[5]{\mathbf g^{\parent{#1, #4, #5}}_{#2, #3}} 
\newcommand{\principalgalphapart}[6]{\mathrm g^{\parent{#1, #4, #5},#6}_{#2, #3}} 

\newcommand{\galphaerror}[5]{\mathbf \calG^{\parent{#1, #4, #5}}_{#2, #3}}
\newcommand{\galphaerrorpart}[6]{\mathrm \calG^{\parent{#1, #4, #5},#6}_{#2, #3}}

\newcommand{\Fgamma}[4]{F^{\parent{#1, #4}}_{#2, #3}} 
\newcommand{\Agamma}[5]{A^{\parent{#1, #5}}_{#2}\parent{#3, #4}} 

\newcommand{\Ggamma}[4]{G^{\parent{#1, #4}}_{#2, #3}} 
\newcommand{\Bgamma}[5]{B^{\parent{#1, #5}}_{#2}\parent{#3, #4}} 

\newcommand{\indexmatrix}{\lambda} 
\newcommand{\errormatrix}{\mu} 

\newcommand{\projection}{\phi}

\newtheorem{thm}{Theorem}[section]
\newtheorem{theorem}[thm]{Theorem}

\newtheorem{lemma}[thm]{Lemma}
\newtheorem{prop}[thm]{Proposition}
\newtheorem{proposition}[thm]{Proposition}
\newtheorem{conj}[thm]{Conjecture}
\newtheorem{question}[thm]{Question}

\theoremstyle{definition}
\newtheorem{definition}[thm]{Definition}

\newtheorem*{example}{Example}

\newtheorem*{remark}{Remark}

\newcommand{\parent}[1]{ \left( #1 \right) }

\newcommand{\inv}{^{-1}}

\newcommand{\SL}{{\text {\rm SL}}}
\newcommand{\GL}{{\text {\rm GL}}}

\newcommand{\set}[1]{\left\{ #1 \right\}}
\newcommand{\prm}{^\prime}

\renewcommand{\conj}[1]{{\overline{#1}}}

\newcommand{\im}{\operatorname{Im}}

\renewcommand{\i}{\operatorname{i}}
\newcommand{\e}{\operatorname{e}}
\newcommand{\smatrix}[4]{\left(\begin{smallmatrix}#1&#2\\#3&#4\end{smallmatrix}\right)}

\begin{document}
\title{The arithmetic of modular grids}
\author{Michael Griffin}
\author{Paul Jenkins}
\author{Grant Molnar}
\thanks{The first author acknowledges the support of the NSF (grant DMS-$1502390$). This work was partially supported by a grant from the Simons Foundation ($\# 281876$ to Paul Jenkins). The third author is grateful for the support of the Gridley Fund for Graduate Mathematics.}
\date{\today}

\begin{abstract}
A modular grid is a pair of sequences $\parent{f_m}_m$ and $\parent{g_n}_n$ of weakly holomorphic modular forms such that for almost all $m$ and $n$, the coefficient of $q^n$ in $f_m$ is the negative of the coefficient of $q^m$ in $g_n$.
Zagier proved this coefficient duality in weights $1/2$ and $3/2$ in the Kohnen plus space, and such grids have appeared for Poincar\'{e} series, for modular forms of integral weight, and in many other situations. We give a general proof of coefficient duality for canonical row-reduced bases of spaces of weakly holomorphic modular forms of integral or half-integral weight for every group $\Gamma \subseteq\SL_2(\bbR)$ commensurable with $\SL_2(\bbZ)$.
We construct bivariate generating functions that encode these modular forms, and study linear operations on the resulting modular grids.
\end{abstract}
\maketitle

\section{Introduction and statement of results}

In~\cite{ZagierTSM}, Zagier defined canonical bases for spaces of weight $1/2$ and weight $3/2$ weakly holomorphic modular forms for $\Gamma_0(4)$ in the Kohnen plus space. The weight $1/2$ basis elements \[f_m(z) = q^{-m}+\sum\limits_{n\geq 1}a_{1/2}(m,n)q^n\] and weight $3/2$ basis elements \[g_n(z) = q^{-n}+\sum\limits_{m\geq 0} a_{3/2}(n,m) q^m,\] where as usual $q^n = \e^{2 \pi \i n z}$, have Fourier coefficients satisfying the striking duality \[
a_{1/2}(m,n) = -a_{3/2}(n,m),
\]
\noindent
so that the $n$th Fourier coefficient of the $m$th basis element in one weight is the negative of the $m$th Fourier coefficient of the $n$th basis element in the other weight. This duality is apparent when we compute the Fourier expansions of the first few basis elements.
\begin{align*}
	f_0(z) &= &1 &{}&{}+{} 2q &{}+{}& 2q^4 &{}+{}& 0q^5 &{}+{}& 0q^8 &{}+{}& \cdots,\\
	f_3(z) &= &q^{-3} &{}&{}-{} 248q &{}+{}& 26752q^4 &{}-{}& 85995q^5 &{}+{}& 1707264q^8 &{}+{}& \cdots,\\
	f_4(z) &= &q^{-4} &{}&{}+{} 492q &{}+{}& 143376q^4&{}+{}& 565760q^5&{}+{}& 18473000q^8&{}+{}& \cdots,\\
	f_7(z) &= &q^{-7} &{}&{}-{} 4119q&{}+{}& 8288256q^4&{}-{}& 52756480q^5&{}+{}& 5734772736q^8&{}+{}& \cdots,\\
	f_8(z) &= &q^{-8} &{}&{}+{} 7256q&{}+{}&26124256q^4&{}+{}&190356480q^5&{}+{}&29071392966q^8&{}+{}& \cdots.
\\\\
	g_1(z)&=&q^{-1} &-2 &{}+{} 248q^3 &{}-{}& 492q^4 &{}+{}& 4119q^7 &{}-{}& 7256q^8 &{}+{}& \cdots,\\
	g_4(z)&=&q^{-4} &-2 &{}-{} 26752q^3 &{}-{}& 143376q^4 &{}-{}& 8288256q^7 &{}-{}& 26124256q^8 &{}+{}& \cdots,\\
	g_5(z)&=&q^{-5} &+0 &{}+{} 85995q^3 &{}-{}& 565760q^4 &{}+{}& 52756480q^7 &{}-{}& 190356480q^8 &{}+{}& \cdots,\\
	g_8(z)&=&q^{-8} &+0 &{}-{} 1707264q^3 &{}-{}& 18473000q^4 &{}-{}& 5734772736q^7 &{}-{}& 29071392966q^8 &{}+{}& \cdots.
\end{align*}
\noindent
Looking at these expressions, we see that the coefficients appearing in the rows of one basis appear, with opposite sign, in the columns of the other basis. Two such sequences of weakly holomorphic modular forms constitute a \emph{modular grid} (see Section \ref{SecDualityWHMF}).

Zagier gave two proofs of this coefficient duality. The first proof relies on showing that these basis elements have a common two-variable generating function
\[
F(z,\tau)=\frac{f_0(z)g_4(\tau)+f_3(z)g_1(\tau)}{j(4\tau)-j(4z)},
\]
where $j$ is Klein's $j$-function. This function encodes the modular grid, so that the Fourier coefficients of $F(z,\tau)$ in one variable yield the basis elements as functions in the other variable. Similar generating functions for bases of modular forms have been used by Eichler~\cite{Eichler1, Eichler2}, Borcherds~\cite{Borcherds1, Borcherds2}, and Faber~\cite{Faber1, Faber2} in studying the Hecke algebra and the zeros of modular forms. The second proof given by Zagier, attributed to Kaneko, is simpler and involves considering, on the one hand, the constant term of the product $f_m(z) g_n(z)$, which is $a_{1/2}(m,n) + a_{3/2}(n,m)$. On the other hand, $f_m(z) g_n(z)$ lives in a space of weight $2$ forms which are derivatives, and so this constant term is equal to $0$.

Modular grids were found for weakly holomorphic modular forms over $\SL_2(\Z)$ of small integer weights by Asai, Kaneko, and Ninomiya in 1997~\cite{Asai-Kaneko-Ninomiya}. Coefficient duality was later proven for a wide variety of spaces of modular forms of integral weight~\cite{Adams, Cho-Choi-Kim, Cho-Choie, Choi-Kim, Duke-Jenkins, El-Guindy, Garthwaite-Jenkins, Guerzhoy, Haddock-Jenkins, Iba-Jenkins-Warnick, Jenkins-Molnar, Jenkins-Thornton2, JKK, Rouse, VanderWilt} and half-integral weight~\cite{Duke-Jenkins2, Green-Jenkins, Kim}.
In many cases, Zagier's and Kaneko's proofs have been modified in ad hoc ways to obtain duality for particular spaces of modular objects. In general, generating functions become increasingly complicated and difficult to find in the presence of cusp forms, while the existence of weight 2 Eisenstein series in higher levels allows for nonzero constant terms in spaces of weight 2, which interfere with Kaneko's proof in these settings.

Harmonic Maass forms provide an alternative path to coefficient duality. The Maass--Poincar\'e series form bases for spaces of harmonic Maass forms with poles at specified cusps. Explicit formulas for their coefficients~\cite{Bringmann-Kane-Rhoades} demonstrate that duality holds in this setting as well, creating ``mock modular grids''. Duality also holds in spaces of vector valued forms, mock theta functions, and other modular objects~\cite{Ahlgren-Kim, Bringmann-Jenkins-Kane, Bringmann-Ono, Cho-Choi-Kim, Choi-Lim, Folsom-Ono, Zhang1}. However, these constructions of coefficients as infinite sums, which turn out to be closely related for dual weights, seem to give less insight into how the spaces of modular forms interact with each other than the proofs of Zagier and Kaneko.

Earlier results related to coefficient duality for the principal parts of weakly holomorphic modular forms of negative weight and coefficients of holomorphic forms of positive weight appear in work of Petersson~\cite{Petersson} and Siegel~\cite{Siegel}. Additionally, unpublished work of Knopp in the mid-1960s (see~\cite{Pribitkin2}) restated Lehmer's conjecture on the coefficients $\tau(n)$ of the $\Delta$-function as a statement about coefficients of forms of weight $-10$.

In this paper, we generalize both Zagier's and Kaneko's proofs for coefficient duality to a very general setting, without relying on explicit formulas for coefficients of the Maass-Poincar\'e series. In particular, suppose $k\in \frac12\Z$, and let $\Gamma\subseteq\SL_2(\R)$ be any subgroup which is commensurable with $\SL_2(\Z)$; note that $\Gamma$ will necessarily be discrete. Moreover, let $\nu$ be a consistent finite weight $k$ multiplier for $\Gamma$, as in Section \ref{SecModularity and cusps} below. We let $\spcfinf k \parent{\Gamma, \nu}$ denote the space of weight $k$ weakly holomorphic modular forms for $\Gamma$ and $\nu$ with poles allowed only at the cusp $\cusp \infty$, and let $\spcginf k (\Gamma,\nu)\subseteq \spcfinf k (\Gamma,\nu)$ be the subspace of weakly holomorphic modular forms which vanish at every cusp besides $\cusp \infty$.   As usual, we write $M_k^!(\Gamma,\nu)$, $M_k (\Gamma,\nu)$, and $S_k(\Gamma,\nu)$ for the related spaces of weakly holomorphic modular forms, holomorphic modular forms, and cusp forms respectively.

\begin{remark}The notation $S^\sharp_k$ is used to denote spaces analogous to $\spcginf k \parent{\Gamma, \nu}$ in~\cite{Jenkins-Molnar} and other papers.
We have altered the notation to avoid confusion with $S^!_k$, which represents the space of weakly holomorphic cusp forms which have vanishing constant terms at \emph{every} cusp.
\end{remark}

The multiplier $\nu$ governs which powers of $q$ can occur in the Fourier expansions of modular forms for $\Gamma$ and $\nu$, via the image of $\nu$ on the translations in $\Gamma$. Let $\Znu \nu $ be defined as in \eqref{Equation: DefZnu} below; this is an indexing set of rational numbers with fixed denominators and with numerators satisfying certain congruence conditions. We note that $\overline\nu$ is a multiplier for $\Gamma$ of weight $2-k$. Moreover, $\Znu {\conj \nu}$ will satisfy congruence conditions opposite to $\Znu \nu$, so that $\Znu {\conj \nu}=-\Znu \nu$. In Section \ref{SecDualityWHMF} we define row-reduced canonical bases for $M^{(\infty)}_{k}(\Gamma,\nu)$ and $\widehat M^{(\infty)}_{k}(\Gamma,\nu)$, indexed by the sets $\tI \nu k$ and $\tJ \nu k$, whose Fourier expansions are given respectively by
\begin{align*}
	\finf {\nu} {k} m(z) &= q^{-m} + \sum_{\substack {n \in \Znu \nu \\ -n \not\in \tI \nu k}} \ainf {\nu} {k} mn q^n \ \ \textrm{for} \ \ m \in \tI {\nu} {k}, \\
	\ginf \nu k n (z) &= q^{-n} + \sum_{\substack {m \in \Znu \nu \\ -m \not\in \tJ \nu k}} \binf \nu k nm q^m \ \ \textrm{for} \ \ n \in \tJ \nu k.
\end{align*}
\noindent
With this notation, our principal theorem, proven in Section~\ref{SecDualityWHMF}, is as follows.
\begin{theorem}\label{IntroThmDualitySimple}
Let $\Gamma$, $k,$ and $\nu$ be as above. For every $m \in \tI \nu k$ and every $n \in \tJ \nu k$, the coefficients of the forms $\finf \nu k m(z)$ and $\ginf {\conj \nu} {2-k} n(z)$ satisfy
\[
	\ainf { \nu} k mn = - \binf {\conj \nu} {2-k} nm.
\]
\end{theorem}
This theorem includes the coefficient duality results in~\cite{Adams, Duke-Jenkins, Garthwaite-Jenkins, Haddock-Jenkins, Iba-Jenkins-Warnick, Jenkins-Molnar, Jenkins-Thornton1, Jenkins-Thornton2, VanderWilt} as corollaries. We generalize our duality results in Section~\ref{SecGeneralizations} to forms with different vanishing conditions (see Theorem \ref{ThmVanishingDuality}), and to forms with poles at other cusps, or (almost equivalently) for the expansions of $\finf \nu k m$ and $\ginf {\conj \nu} {2-k} n$ at other cusps (see Theorem \ref{ThmDualityPolesAnywhere}).

Our proof of Theorem \ref{IntroThmDualitySimple} shows that coefficient duality in modular grids can be seen as a consequence of Serre duality by means of the modular pairing defined in Section \ref{Pairing}. Borcherds used a formal version of the pairing, along with Serre duality, to describe the permissible principal parts of weakly holomorphic modular forms in~\cite{Borcherds}, and Bruinier and Funke described the pairing analytically in~\cite[Proposition 3.5]{Bruinier-Funke}. The modular pairing can be interpreted as the sum of residues of a meromorphic differential on the modular curve, coming from the product of the two paired forms; since the sum of such residues must be zero, this gives a linear relation on the coefficients of the forms. Borcherds's work implies coefficient duality within the principal parts of the basis elements, but additional linear algebra and applications of the pairing are required to prove the full statement of duality. This method generalizes Kaneko's proof in a way that applies equally well to weakly holomorphic modular forms of arbitrary level and multiplier.

\begin{remark}
The generality of this theorem relies on the generality of Borcherds's theorem \cite[Theorem 3.1]{Borcherds} (see Theorem \ref{ThmFunkePairingFormExistence}), which is written in terms of vector valued modular forms, and so our results here could be easily adapted to that setting. Other generalizations, such as to modular forms with real weight or with multipliers with infinite image, would require an appropriate generalization of Theorem \ref{ThmFunkePairingFormExistence}.
\end{remark}

Taking inspiration from Zagier's proof of duality, we give an algorithm in Section~\ref{SecGeneratingFunctions} to construct the generating functions which encode the modular grids for spaces of weakly holomorphic forms.

\begin{theorem}\label{ThmGenFn}
Let $\Gamma$, $k,$ and $\nu$ be as above, and let $\parent{\finf \nu k m(z)}_m$ and $\parent{\ginf {\conj \nu} {2-k} n(z)}_n$ be the row-reduced canonical bases for the spaces $\spcfinf k \parent{\Gamma, \nu}$ and $\spcginf {2-k} \parent{\Gamma, \conj \nu}$. There is an explicit generating function $\calH_{k}^{(\nu)}(z, \tau)$ encoding this modular grid; that is, on appropriate regions of $\H \times \H$, the function $\calH_{k}^{(\nu)}(z, \tau)$ satisfies the two equalities
\[
	\calH_{k}^{(\nu)}(z, \tau) = \sum_{m \in \tI \nu k} \finf \nu k m(z) \e^{2\pi \i m \tau} \ \text{and} \ \calH_{k}^{(\nu)}(z, \tau) = -\sum_{n \in \tJ {\overline \nu} {2-k}} \ginf {\conj \nu} {2 - k} n(\tau) \e^{2\pi \i n z}.
\]
The function $\calH_{k}^{(\nu)}(z, \tau)$ is meromorphic on $\H \times \H$ with simple poles when $z = \gamma \tau$ for some $\gamma \in \Gamma$. It is modular in $z$ of weight $k$ with multiplier $\nu$, and modular in $\tau$ of weight $2 - k$ with multiplier $\overline{\nu}$.
\end{theorem}

\begin{remark}
The technique described in the proof of Theorem \ref{ThmGenFn} can be used to construct the generating functions given in~\cite{Adams, Ahlgren, Atkinson, Duke-Jenkins, Garthwaite-Jenkins, Haddock-Jenkins, Jenkins-Molnar, VanderWilt, Ye, ZagierTSM}, as well as many of the generating functions found in~\cite{El-Guindy} for sequences of modular forms that form canonical bases for appropriate spaces.
\end{remark}

New modular grids can also be constructed through various transformations and linear operations on existing modular grids. We describe three such operations in Section~\ref{SecArithmetic}.

When $M_k(\Gamma,\nu)$ and $S_{2-k}(\Gamma,\overline\nu)$ are both trivial, the generating functions $\calH_{k}^{(\nu)}(z, \tau)$ described in Theorem \ref{ThmGenFn} exhibit surprising symmetry under some of these linear operations.  For example, we show in Theorems \ref{ThmMatrixSymm} and \ref{ThmTrace} that, under these conditions, the action by a matrix or trace operation in one variable corresponds with a similar action in the other variable. Letting $|_{\kappa,w}T_M$ denote the weight $\kappa$ action of the Hecke operator in the variable $w$, normalized as in \eqref{HeckeDef}, we have the following theorem relating the Hecke operators in variables of dual weights.

\begin{theorem}\label{HeckeSym}
 Let $k$ be a positive even integer and $\Gamma\supseteq\Gamma_0(N)$ for some $N\in \N$. Moreover, let $\nu$ be a multiplier which is trivial on
 $\Gamma_0(N)$, such that $S_{k}(\Gamma,\overline\nu)=\{0\}.$ If $M$ is a positive integer coprime to $N$, then
 \[
\finfgen \nu {2-k} (z, \tau)\big|_{k,\tau}T_{M}= \finfgen \nu {2-k} (z, \tau)\big|_{2-k,z}T_{M}.
 \]
\end{theorem}

\section*{Acknowledgments}
We thank John Voight and Pavel Guerzhoy for helpful discussions, and we thank the anonymous referee for many thoughtful suggestions which helped improve this paper.

\section{Modular groups and cusps} \label{SecModularity and cusps}

In this section, we give background on modularity, certain useful generalizations of the metaplectic group, commensurable subgroups, and functions which are modular over these groups. We also describe how the Fourier coefficients of those functions at each cusp may be defined.
\subsection{Modularity and generalizations of the metaplectic group }

Following Kohnen \cite{Kohnen}, for $k \in \R$ we define $\mathfrak G_{k}$ to be the set of pairs $(\gamma,\phi)$, where $\gamma=\smatrix abcd\in \GL_2^+(\R)$ and $\phi$ is a holomorphic function on the upper half plane, satisfying $|\phi(z)|=\det( \gamma)^{-k/2}|cz+d|^{k}.$
This set has a group operation given by
 \begin{equation}\label{gp_law}
\big(\gamma,\phi(z)\big)\big(\mu,\psi(z)\big) \, =\, \big(\gamma\mu, \  \phi(\mu z)\psi(z)\big),\end{equation}
where $\smatrix abcd z=\frac{az+b}{cz+d}.$
We will be primarily interested in the case $k\in \frac12\Z$.  Note that if $k=\frac12$, $\mathfrak G_k$ is
an extension of the usual metaplectic group $\textrm{Mp}_2(\R)$, which is defined similarly except that we require $\gamma=\smatrix abcd\in \SL_2(\R)$ and $\phi(z)=\pm\sqrt{cz+d}$.

Given a function $f: \H \to \C$ and $(\gamma,\phi)\in \mathfrak G_k$, the action of the weight $k$ Petersson slash operator is defined by
\begin{align}\label{SlashDefM2}
	f(z)|_{k}(\gamma,\phi) = \phi(z)^{-1} f(\gamma z).
\end{align}
If $\mathbf \Gamma$ is any subgroup of $\mathfrak G_k,$ we say that $f$ is modular of weight $k$ for
$\Gamma$ if $f(z)|_{k}(\gamma,\phi)=f(z)$ for every $(\gamma,\phi)\in \mathbf\Gamma.$

The subgroups $\mathbf \Gamma\subseteq \mathfrak G_k$ for which there are nonzero functions $f$ which are modular for $\mathbf \Gamma$ satisfy certain properties purely as algebraic consequences of the group action.
\begin{proposition}\label{MetaplecticProp} Let $\mathbf \Gamma\subseteq \mathfrak G_k$ be a subgroup such that there is a function $f$
which is  modular for $\mathbf \Gamma$ and nonzero at some point $z_0\in \H$. Then the following are true.
\begin{enumerate}
\item The projection $(\gamma,\phi)\to \gamma$ is a group isomorphism $\mathbf \Gamma\to \Gamma$ for some group $\Gamma\subseteq \GL_2^+(\R)$.\label{Item: projection is an isomorphism}

\item For each $\gamma\in \Gamma$ defined in part \eqref{Item: projection is an isomorphism}, there is a unique $\phi$ such that  $(\gamma,\phi)\in \mathbf \Gamma.$
\item If $(\gamma,\phi)\in \mathbf \Gamma$ and $\gamma=\smatrix a00a$ is a scalar matrix, then $\phi=1.$\label{Item: scalar matrices yield phi = 1}

\end{enumerate}
\end{proposition}
\begin{remark}
Let $\mathfrak S_k=\{(\gamma,\phi)\in \mathfrak G_k \ : \ \gamma\in \SL_2(\R) \},$ and let $\mathfrak I=\{(\smatrix a00a, 1) \ : \ a>0\}.$
In light of part \eqref{Item: scalar matrices yield phi = 1} of Proposition \ref{MetaplecticProp}, we generally restrict our attention to subgroups $\mathbf \Gamma \subseteq\mathfrak S_k \simeq \mathfrak G_k /\mathfrak I$. The reason for taking the quotient by only the {positive} scalar matrices (at least when $k\in \frac12 +\Z$) will be apparent later as we define multipliers. Often we take advantage of this isomorphism to allow us to write certain matrices (such as the Fricke involution in example \ref{ExFricke} below) as integer matrices, rather than using square roots or fractions.
\end{remark}

\begin{proof}
\begin{enumerate}
	\item Note that the definition of the group law implies that the projection $(\gamma,\phi)\to \gamma$ is a group homomorphism to $\GL_2^+(\R)$. The projection is clearly surjective onto its image. Thus, to complete the proof of part (1), it suffices to prove part (2).
	
\item Suppose $(\gamma,\phi), (\gamma,\psi)\in \mathbf \Gamma.$ Then $(\gamma,\psi(z))^{-1}=(\gamma^{-1},\psi(\gamma z)^{-1}),$ and so
	\[(\gamma,\phi(z)) (\gamma,\psi(z))^{-1} = \left(I, \phi(\gamma z)\psi(\gamma z)^{-1}\right).
	\]
	Thus, in order to show $\phi=\psi,$ it suffices to prove part (3).

\item Suppose $(\gamma,\phi)\in \mathbf \Gamma$ with $\gamma=\smatrix a00a$ for some $a\in \R.$ By the definition of $\mathfrak G_k$, we have that  $\phi$ is a holomorphic function with $|\phi(z)| = 1$, and is therefore constant. If $f: \H \to \C$ is modular for $\mathbf \Gamma,$  and $f(z_0)\neq 0$, then
\[\phi(z_0)^{-1} f(z_0) =\phi(z_0)^{-1} f(\gamma z_0)=f|_{k}(\gamma,\phi)(z_0)=f(z_0).\] Therefore $\phi(z_0)=1$.
\end{enumerate}
\end{proof}

Given $\gamma = \smatrix abcd \in\GL_2^+(\R),$ let
\[
	j(\gamma,z)=\det(\gamma)^{-1/2}(cz+d).
\]
If $k\in \frac12 \Z$, then part (2) of Proposition \ref{MetaplecticProp} implies that any such $\mathbf \Gamma$ is of the form $\{(\gamma,\nu (\gamma) j(\gamma,z)^{k}\ : \ \gamma\in \Gamma\}$ for some $\Gamma\subseteq \GL_2^+(\R)$ and some well-defined map $\nu: \Gamma \to \C^\times.$  Here we take the principal branch of the square root when $k\in \frac12+\Z$.

\begin{definition}
Suppose $\Gamma$ is a discrete subgroup of $\GL_2^+(\R)$ and $k\in \frac12\Z$. If $\nu:\Gamma\to \C^\times$ is a function such that $\{(\gamma, \nu(\gamma)j(\gamma, z)^k)\ : \ \gamma\in \Gamma\}$ is a subgroup of $\mathfrak G_k$, and $\nu(\gamma){j(\gamma,z)}^{k}=1$ for all scalar matrices $\gamma\in\Gamma$, then we say that $\nu$ is a \emph{consistent weight $k$ multiplier}~\cite[Section 2.6]{Iwaniec}. For scalar matrices $\gamma$, this forces $|\nu(\gamma)| = 1$. If additionally the image of $\nu$ is finite, we say that $\nu$ is a \emph{finite multiplier}; in this case $|\nu(\gamma)| = 1$ for all $\gamma$.
\end{definition}
\begin{proposition}\label{nuProp}
Let $\Gamma$ be a discrete subgroup of $\GL_2^+(\R),$ $k\in \frac12 \Z,$ and $\nu$ a consistent weight $k$ multiplier. Then the following are true.
\begin{enumerate}
\item If $\gamma,\mu\in \Gamma,$ then
\[
\nu(\gamma\mu)=\pm \nu(\gamma)\nu(\mu),
\]
where the sign is determined by~\eqref{gp_law}. If $k\in \Z,$ then the sign is positive, so $\nu$ is a character.
\item If $k'\in \Z,$ then $\nu$ is also a consistent weight $k+k'$ multiplier.
\item The function $\overline \nu$ is a consistent weight $-k$ multiplier and thus a weight $2-k$ multiplier.
\item If $\gamma=\smatrix a00a\in \Gamma$ is a scalar matrix, then $\nu(\gamma)=\begin{cases} 1& \text{ if } a>0\\ (-\i)^{2k}& \text{ if } a<0.\\\end{cases}$\label{Item: nu of scalar matrices}
\end{enumerate}
\end{proposition}
\begin{proof}
The first three are simple consequences of the group law \eqref{gp_law}. Part \eqref{Item: nu of scalar matrices} follows from Proposition \ref{MetaplecticProp} (3) since in this case $\nu(\gamma)j(\gamma,z)^{k} = 1,$ and $j(\gamma,z)=\text{sgn}(a).$
\end{proof}

In the remainder of this paper, we will generally describe modularity in terms of a subgroup $\Gamma\in \GL_2^+(\Z)$ and a multiplier $\nu,$ rather than in terms of a subgroup $\mathbf \Gamma \subseteq \mathfrak G_k.$ Given a function $f: \H \to \C$, the action of the weight $k$ Petersson slash operator for an element $\gamma \in \GL_2^+(\R)$ is given by
\begin{align}\label{SlashDef}
	f(z)|_{k}\gamma = j(\gamma,z)^{-k} f(\gamma z).
\end{align}
If $\nu$ is a consistent weight $k$ multiplier for some group $\Gamma$, then we define
\[
	f(z)|_{k,\nu}\gamma
	 = \nu(\gamma)^{-1}j(\gamma,z)^{-k} f(\gamma z).
\]
We say $f$ is modular of weight $k$ for $\Gamma$ with multiplier $\nu$ if $f(z)|_{k,\nu}\gamma=f(z)$ for every $\gamma\in \Gamma.$

\begin{example}\label{ExFricke}
Let $N\in \N$, and let $\Gamma_0^+(N)$ be the group generated by $\Gamma_0(N)$ and the Fricke involution $w_N=\smatrix0{-1}N0.$ If $k\in \frac12\Z$ and $\nu$ is any consistent weight $k$ multiplier for $\Gamma_0^+(N),$ then there are only two possibilities for $\nu(w_N).$  Using \eqref{gp_law} to square the element $(w_N, \nu(w_N)N^{-k/2}(Nz)^k),$ we find that
\[
\nu(w_N)^2=\nu(-NI)=\nu(-I)=(-\i)^{2k},
\]
where $I$ is the identity matrix. Hence $\nu(w_N)=\pm(-\i)^k$. In particular, we conclude that if a function $f$ is modular for $\Gamma_0(N)$ and is also an eigenfunction for the Fricke involution, then it has one of two possible eigenvalues, which are determined by the weight.

For instance, the group $\Gamma_0^+(4)$ is generated by the matrices $\smatrix1101,$ $\smatrix1041,$ $-I$ and $w_4.$
The standard theta function $\theta(z)=\sum_{n\in \Z} q^{n^2}$ is modular of weight $1/2$ for $\Gamma_0^+(4)$ with a multiplier $\nu$ which is trivial on $\smatrix1101$ and $\smatrix1041,$ while $\nu(-I)=-\i$ and $\nu(w_4)=\sqrt{-\i}$.

\end{example}

\subsection{Commensurable subgroups}

Let $G$ and $H$ be subgroups of a group $K$. We say that
$G$ is commensurable with $H$ if \[
[G: G\cap H] < \infty \ \text{and} \ [H : G \cap H] < \infty.
\]
In this paper we are interested in groups $\Gamma\subseteq\SL_2(\R)$ that
are commensurable with $\SL_2(\Z)$. For the remainder of this paper, we will simply refer to such groups as \emph{commensurable subgroups.}

Borel showed that every commensurable subgroup $\Gamma$ is contained in a \emph{maximal} commensurable subgroup of a very specific shape.

\begin{thm}[{\cite[Proposition 4.4 (iii)]{Borel}}]\label{ThmBorel}
	Let $\Gamma$ be a commensurable subgroup. Then $\Gamma$ is conjugate to a subgroup of
$\Gamma^*_0(N)$ for some square-free integer $N$. Here $\Gamma^*_0(N)$ is the group generated by $\Gamma_0(N)$ and all the Atkin--Lehner involutions $w_d$ with $d$ dividing $N.$
\end{thm}

\begin{remark}
The choice of a square-free $N$ in Theorem~\ref{ThmBorel} for a group $\Gamma$ is not necessarily unique; for instance, $\Gamma_0(p)$
is contained in both $\Gamma_0^*(p)$ and $\SL_2(\Z)$. Additionally, if $\Gamma$ is a congruence subgroup, then $N$ will always divide the level of $\Gamma$, but we may not always be able to take it to be the largest square-free divisor of the level; for instance, $\Gamma=\langle \Gamma_0(p^2),w_{p^2}\rangle$ is contained in $\smatrix 100p\SL_2(\Z)\smatrix 100p^{-1}$ but not in any conjugate of $\Gamma_0^*(p).$
\end{remark}

Given a commensurable subgroup $\Gamma$, the modular curve $Y(\Gamma)$ is defined by $ Y(\Gamma)= \Gamma\backslash\H$. The compactified modular curve $X(\Gamma)$ is a compact Riemann surface obtained by adding a finite collection of points $\cusps \Gamma$, the cusps of $\Gamma$, to $Y(\Gamma)$. Let $\Q^* = \Q \cup \set \infty$ and $\H^* = \H \cup \Q^*$. We define
\begin{align*}
	X(\Gamma) &= \Gamma\backslash\H^* \quad \quad \text{and} \quad \quad
\cusps \Gamma = \Gamma\backslash \Q^*.
\end{align*}
Thus $X(\Gamma) = Y(\Gamma) \sqcup \cusps \Gamma$.

Fix $\widehat \Gamma$ a maximal commensurable subgroup over $\Gamma$; there is a rational map
from $X(\Gamma)$ to $X(\widehat \Gamma)$ given by $\Gamma z \mapsto \widehat{\Gamma} z$. Since $\widehat \Gamma$ is conjugate to $\Gamma^*_0(N)$ with $N$ square-free, $X(\widehat \Gamma)$ has exactly one cusp. Therefore, given any cusp $\cusp \rho$ of $\Gamma$, there is is some $\gamma_\rho\in \widehat \Gamma$ so that $\cusp {\gamma_\rho \infty} = \cusp \rho.$ As a notational convenience, for each $\cusp \rho \in \cusps \Gamma$ we fix $\gamma_\rho \in \widehat{\Gamma}$ with $\cusp{\gamma_\rho \infty} = \cusp \rho$.

\subsection{Expansions at cusps}\label{Sec:Expansions}

Let $\Gamma$ be a commensurable subgroup contained in a maximal commensurable subgroup $\widehat \Gamma$.  Let $k$ be an integer or half integer, and let $\nu$ be a consistent finite weight $k$ multiplier for $\Gamma$.  The group $\Gamma$ contains a maximal subgroup $\Gamma_\infty$ of upper triangular matrices.   Since $\Gamma$ is commensurable with $\SL_2(\bbZ)$, the intersection $\Gamma \cap \SL_2(\bbZ)$ is of finite index in $\Gamma$.  Therefore, given an upper triangular matrix $\gamma \in \Gamma_\infty$, we must have $\gamma^n \in \Gamma \cap \SL_2(\bbZ)$ for some integer $n$.  This forces the diagonal entries to be real roots of unity, and we conclude that every element of $\Gamma_\infty$ is of the form $\pm \smatrix 1r01$.  Thus $\Gamma_\infty$ is a finitely generated abelian subgroup of $\SL_2(\bbR)$ of rank 1.  But as $\Gamma_\infty$ is discrete, the homomorphism $\projection : \Gamma_\infty \to \bbQ$ given by $\projection : \pm \smatrix 1r01 \mapsto r$ satisfies $\projection\parent{\Gamma_\infty} = w \bbZ$ for some positive rational number $w$. We set $w(\nu) = w$, and say $w(\nu)$ is the \emph{width} of the cusp $\cusp \infty$. Fix $T_\Gamma \in \projection\inv(w)$, so $T_\Gamma = \pm \smatrix 1w01$; for simplicity, we choose $T_\Gamma$ to be $\smatrix 1w01$ where possible. Since the image of $\nu$ is finite,
$\nu \parent{T_\Gamma} = \e^{2\pi \i \xi}$ for some rational $\xi$.
In consequence, if $f$ is holomorphic on $\H$ and is modular for $\Gamma$ and $\nu$ of weight $k$, then $f$
satisfies
\[f(z)|_{k,\nu}T_\Gamma = \begin{cases} e^{-2\pi \i \xi} f(z+w) = f(z) &\textrm{ if } T_\Gamma = \smatrix 1w01, \\
e^{-2 \pi \i \xi} \i^{-2k} f(z+w) = f(z) &\textrm{ if } T_\Gamma = \smatrix{-1}{-w}{0}{-1}.\end{cases}\]
Define $\varsigma = \varsigma(\nu)$ to be the rational number satisfying $0 \leq \varsigma < 1$ such that for any $f$ that is holomorphic on $\H$ and modular of weight $k$ for $\Gamma$ and $\nu$, we have
\[f(z+w) = e^{2\pi \i \varsigma}f(z),\]
so that $f(z)$ has a Fourier expansion of the form
\begin{equation*}
	f(z) = \sum_{n\in \Z}a\parent{\frac{n+\varsigma}{w}}q^{\frac{n+\varsigma}{w}}.
\end{equation*}

In many cases, $\varsigma \equiv \xi \pmod{1}$.  However, when $T_\Gamma = \smatrix{-1}{-w}{0}{-1}$, the matrix $\smatrix 1w01$ and the negative identity matrix are not in the group $\Gamma$, and  we say that $\infty$ is an \emph{irregular cusp}, as seen in~\cite[pp. 74--75 and errata]{DS} for the odd weight case.  In this case, $\varsigma \equiv \xi + k/2 \pmod{1}$.

We can rewrite the summation of the exponents in terms of a congruence condition as follows.  If $\varsigma = \frac st$ in lowest terms, so that $s = s(\nu)$ and $t = t(\nu)$ are coprime integers with $t > 0$, then we define the set
\begin{equation}\label{Equation: DefZnu}
	\Znu \nu = \set{\frac{m}{wt} \ \ : \ \ \ m\in \Z, \ m\equiv s \pmod{t}}
\end{equation}
so that
\begin{equation}\label{Equation: ExpInf}
	f(z) = \sum_{n\in \Znu \nu}a(n)q^n.
\end{equation}
We refer to \eqref{Equation: ExpInf} as the Fourier expansion of $f$ at the cusp $\cusp \infty$; note that the difference between adjacent elements in $\Znu{\nu}$ is the width $w$. Since $\overline{\nu}(T_\Gamma) = \e^{-2\pi \i \varsigma}$, we deduce that $\Znu {\conj\nu}=-\Znu \nu$.

We can also expand $f$ about any other cusp; however, the expansion may not be canonical. If $\alpha \in \widehat \Gamma$ and $f$ is as above, then the function $f^\alpha(z) = f(z)|_{k}\alpha$ is modular for the group
\begin{equation}\label{Equation: Gamma^alpha}
	\Gamma^\alpha = \alpha^{-1}\Gamma\alpha
\end{equation}
with multiplier $\nu^\alpha$ defined so that $\nu^\alpha(\alpha^{-1}\gamma\alpha)=\nu(\gamma).$ We note in passing that $\parent{\nu^{\alpha}}^\beta = \nu^{\alpha \beta}$. Since $\Gamma^\alpha$ is still commensurable and $\nu^\alpha$ has finite image, we have that
$f^\alpha(z)$ has a Fourier expansion of the form
\begin{equation} \label{Equation: SmoothFourier}
	f^\alpha(z) = \sum\limits_{n \in \Znu {\nu^\alpha}} a^\alpha(n) q^{n}.
\end{equation}
This is an expansion of $f$ about the cusp $\cusp {\alpha \infty}.$ If $\beta \in \widehat \Gamma$ is any other choice of matrix with $\cusp{\alpha \infty} = \cusp{\beta \infty},$ then $\alpha = \errormatrix \beta T$ for some $\errormatrix \in \Gamma$ and some $T \in \widehat \Gamma$ an upper-triangular matrix. Thus
\[
	f^\alpha = \nu(\errormatrix) f^\beta|_k T.
\]
Since $T\in \widehat \Gamma$, which is a commensurable subgroup, $T$ must have the form $\pm \smatrix1r01$ with $r$ rational, and so the action of the slash operator simply sends $z \mapsto z+r$  and $q^n \mapsto e^{2 \pi \i r}q^n$. In particular, the indexing set $\Znu {\nu^\alpha}$ depends only on the cusp $\cusp {\alpha \infty}$, and not the specific choice of $\alpha$.

We therefore define
\begin{align*}
	\Zrho \nu \rho = \Znu {\nu^{\gamma_\rho}}.
\end{align*}
We also define $\cuspwidth \Gamma \rho = w\parent{\nu^{\gamma_\rho}}$, and say $\cuspwidth \Gamma \rho$ is the \emph{width} of the cusp $\cusp \rho$; by definition, $\cuspwidth \Gamma \rho$ is the unique positive rational number $w$ such that $\projection\parent{\Gamma_\infty^{\gamma_\rho}}$ is generated by $w$. These definitions are independent of our choice of $\gamma_\rho \in \widehat\Gamma$, but they do depend on $\widehat \Gamma$. For instance, if $\Gamma=\Gamma_0(p)$ for some prime $p$, them we may choose $\widehat \Gamma$ to be $\SL_2(\Z)$ or $\Gamma^*_0(p)$. In the former case, $\cuspwidth \Gamma 0=p$, and we may take $\gamma_0=\left(\begin{smallmatrix} 0&-1\\1&0\end{smallmatrix}\right)$. In the latter case, $\cuspwidth \Gamma 0=1$, and we may take $\gamma_0$ to be the Fricke involution $\gamma_0=\left(\begin{smallmatrix} 0&-1\\p&0\end{smallmatrix}\right)$.

The lemma below follows immediately from the discussion above.
\begin{lemma}\label{Lemmamatricesforcusp}
	Suppose $f(z)$ is weakly holomorphic and modular for $\Gamma$ and $\nu$ of some weight, and assume the notation of \eqref{Equation: SmoothFourier}. If $\alpha, \beta \in \widehat{\Gamma}$ with $\cusp {\alpha \infty} = \cusp {\beta \infty}$, then $\alpha = \errormatrix \beta T$ for some $\errormatrix \in \Gamma$ and some  $T = \pm \smatrix 1r01 \in \widehat{\Gamma}$. Moreover, for all $n \in \Zrho \nu \rho$ we have
	\[
		a^{\alpha}(n) = \nu(\errormatrix) \e^{2 \pi \i r n} a^{\beta}\parent{n}.
	\]
\end{lemma}
Thus, although the indexing set $\Zrho \nu \rho$ depends only on the cusp $\cusp \rho$, the coefficients $a^\alpha(n)$ and $a^\beta(n)$ may differ by roots of unity.

\section{The modular pairing}\label{Pairing}

Bruinier and Funke~\cite[Eq. (3.9)]{Bruinier-Funke} defined a pairing between real-analytic modular forms of dual weights and conjugate multipliers in terms of the Petersson inner product. If the singularities of the two forms are of meromorphic type, then Proposition 3.5 of~\cite{Bruinier-Funke} shows that the pairing is the sum of residues of the product of the two forms (including the residues at the cusps). For instance, if $f\in M^!_k(\Gamma,\nu)$ and $g\in M^!_{2-k}(\Gamma,\overline \nu)$ have Fourier expansions of the form
\begin{align*}
	f^\lambda(z) = \sum\limits_{n \in \Zrho \nu {\lambda \infty}} a^\lambda(n) q^n \ \ \text{and} \ \ g^\lambda(z) = \sum\limits_{m \in \Zrho \nu {\lambda \infty} } b^\lambda(m) q^m,
\end{align*}
for $\lambda \in \widehat{\Gamma}$, then the pairing $\{f,g\}_\Gamma$ is given as a sum over the set of cusps $\Omega(\Gamma)$ by
\begin{align*}
	\{f,g\}_{\Gamma}=\sum_{\rho\in \Omega(\Gamma)}\frac{\cuspwidth \Gamma \rho}{\widehat {w}} \sum_{\ell \in \Q} a ^{\gamma_\rho}(\ell) b^{\gamma_\rho}(-\ell),
\end{align*}
where $\widehat w$ is the width of the (unique) cusp of $\widehat \Gamma$.
By Lemma \ref{Lemmamatricesforcusp}, $\{f,g\}_{\Gamma}$ is independent of our choice of $\gamma_\rho$; in Bruinier's and Funke's work, $\{f,g\}_{\Gamma}$ is also independent of our choice of $\widehat{\Gamma}$.
The form $f\cdot g$ is a weight $2$ modular form for $\Gamma$ with trivial multiplier, and therefore defines a meromorphic differential on $X(\Gamma)$. It is well-known that the sum of residues of a meromorphic differential on a compact Riemann surface is $0$, and so we have the following theorem.

\begin{theorem}[See Proposition 3.5 of~\cite{Bruinier-Funke} and the observation following the proof]\label{ThmPairing}
Let $\Gamma $ be a commensurable subgroup, let $k\in \frac 1 2 \Z,$ and let $\nu$ be any consistent finite weight $k$ multiplier for $\Gamma$. Suppose $f\in M^!_k(\Gamma,\nu)$ and $g\in M^!_{2-k}(\Gamma,\overline \nu)$. Then
\[
	\{f,g\}_\Gamma = 0.
\]
\end{theorem}

If $g$ is made to range over all holomorphic modular forms in $M_{2-k}^!(\Gamma, \overline{\nu})$, this theorem constrains the possible principal parts that $f$ may have at cusps. The Bruinier--Funke pairing is essentially an analytic interpretation of a formal pairing defined earlier by Borcherds~\cite{Borcherds}. Using Serre duality and this formal version of the pairing, Borcherds proved that these constraints, including the implied constraints given by \eqref{Equation: SmoothFourier}, are the \emph{only} constraints on the principal parts (see Theorem \ref{ThmFunkePairingFormExistence} below).

Suppose $\Gamma$, $\widehat\Gamma $ and $\nu$ are as above, with $w$ and $t$ as in \eqref{Equation: DefZnu}. Let $\C((q))_\nu$ denote the subspace of the field of formal Laurent series $\C((q^{\frac{1}{wt}}))$ consisting of elements $\mathrm h(q)$ of the form
\[
	\mathrm h(q) = \sum_{\substack{n\in \Znu\nu \\n\gg-\infty}} a(n)q^n,
\]
and let
\[
	\C((q))_{\widehat\Gamma,\nu} = \prod^*_{\indexmatrix \in \widehat\Gamma} \C((q))_{\nu^\indexmatrix},
\]
where the direct product is restricted as follows: If $\mathbf h=(\mathrm h^\indexmatrix)_{\indexmatrix \in \widehat\Gamma}\in \C((q))_{\widehat\Gamma,\nu}$ with
\[
	\mathrm h^\indexmatrix (q)= \sum\limits_{\substack{n\in \Zrho \nu {\indexmatrix \infty} \\n\gg-\infty}} a^\indexmatrix(n)q^n,
\]
then whenever $\indexmatrix = \pm \errormatrix \indexmatrix\prm \smatrix1r01$ with $\errormatrix \in \Gamma$ and $r\in \Q$, we have $a^{\indexmatrix}(n) = \nu(\errormatrix) \e^{2\pi \i r n} a^{\indexmatrix\prm}(n).$ This implies that $\mathbf h$ is uniquely determined by a set of components corresponding to any complete set of coset representatives of $\Gamma\backslash \widehat \Gamma\slash \widehat{\Gamma}_\infty,$ which is finite. We regard $\bbC((q))_{\widehat\Gamma, \nu}$ as a formal space of ``candidate'' weakly holomorphic modular forms. By Lemma \ref{Lemmamatricesforcusp}, if ${f\in M^!_k(\Gamma,\nu)}$, then the modular form $f$ has an image in this space given by $f \mapsto (f^\indexmatrix)_{\indexmatrix \in \widehat\Gamma}\in \C((q))_{\widehat\Gamma,\nu},$ and we view $M_k^!(\Gamma, \nu)$ as a subspace of $\C((q))_{\widehat\Gamma,\nu}$ via this embedding. Theorem \ref{ThmFunkePairingFormExistence} below enables us to construct weakly holomorphic modular forms by performing linear algebra on series in $\bbC((q))_{\widehat\Gamma, \nu}$.

Following Borcherds, we define the formal pairing $\{~,~\}_\Gamma :\C((q))_{\widehat\Gamma,\nu}\times \C((q))_{\widehat\Gamma,\overline\nu}\to \C$ by
\begin{align}\label{Equation: Formal modular pairing}
	\{(\mathrm f^\indexmatrix),(\mathrm g^\indexmatrix)\}_\Gamma &= \sum_{\indexmatrix \in \Gamma\backslash \widehat\Gamma}\sum_{~\ell\in \Zrho \nu {\indexmatrix \infty}} a^{\indexmatrix}(\ell)b^{\indexmatrix}(-\ell)=\sum_{\rho\in \Omega(\Gamma)}\frac{\cuspwidth \Gamma \rho}{\widehat w} \sum_{\ell\in \Q} a ^{\gamma_\rho}(\ell)b^{\gamma_\rho}(-\ell)
\end{align}
where
\[
	\mathrm f^\indexmatrix (q)= \sum_{n \in \Zrho \nu {\indexmatrix \infty}} a^\indexmatrix(n) q^n \ \ \text{and} \ \ \mathrm g^\indexmatrix(q) = \sum_{m \in \Zrho \nu {\indexmatrix \infty}}b^\indexmatrix(m) q^m.
\]
Note that the  ratio $\frac{\cuspwidth \Gamma {\rho}}{\widehat w}$ in the right hand side of  \eqref{Equation: Formal modular pairing} appears as we change the set of summation from $\Gamma\backslash \widehat\Gamma$ to $\cusps \Gamma\simeq \Gamma\backslash \widehat\Gamma/\widehat{\Gamma}_\infty.$ If $\lambda$ is any element of $\widehat \Gamma$ with $\lambda\infty=\rho,$  then $\frac{\cuspwidth \Gamma {\rho}}{\widehat w}$ is the number of cosets of $\Gamma\backslash\widehat\Gamma$ contained in $\Gamma\lambda \widehat\Gamma_\infty$.

Naturally, Borcherds's formal pairing depends on $\widehat{\Gamma}$, but it remains independent of our choice of $\gamma_\rho$. Borcherds proved  a vector valued version of the following theorem about this pairing.
\begin{theorem}[see Theorem 3.1 of~\cite{Borcherds}]\label{ThmFunkePairingFormExistence}
 Let $\Gamma $ be a commensurable subgroup, and let $\widehat{\Gamma}$ be a maximal commensurable subgroup containing $\Gamma$. Let $k\in \frac 1 2 \Z,$ and let $\nu$ be any consistent finite weight $k$ multiplier.
Suppose $\mathbf f = \parent{\mathrm f^\indexmatrix}_\indexmatrix \in \C((q))_{\widehat\Gamma,\nu}.$ Then the following are equivalent:
\begin{enumerate}
	\item For every holomorphic modular form $g \in M_{2-k}(\Gamma,\overline\nu)$, we have $\set{\mathbf{f}, g}_\Gamma = 0.$
	\item There exists $f\in M^!_{k}(\Gamma, \nu)$ such that for each $\indexmatrix \in \widehat \Gamma$, we have $f^\indexmatrix = \mathrm f^\indexmatrix + o(1).$
\end{enumerate}
\end{theorem}
\begin{proof}
This is the one-dimensional version of Theorem 3.1 of~\cite{Borcherds}. Here the multiplier $\nu$ is the product of the one-dimensional representation $\rho$ and the sign of $(\sqrt{cz+d})^{2k}$ arising from the action of an element of the metaplectic group. The condition that $\nu$ is finite is equivalent to the condition that $\rho$ factors through a finite quotient of $\Gamma$.
\end{proof}

A version of this theorem also follows readily from Bruinier and Funke's original work, by applying the pairing to the holomorphic parts of Maass-Poincar\'e series. Borcherds's proof relies on Serre duality and avoids harmonic Maass forms and Maass-Poincar\'e series entirely. His method can therefore be applied to circumstances, such as weight $1$, where the convergence of the Maass-Poincar\'e series is less straightforward.

\section{Row-reduced canonical bases}\label{SecDualityWHMF}

Let $I$ and $J$ be sets of rational numbers, and let $\parent{f_m}_{m}$ and $\parent{g_n}_{n}$ be sequences of weakly holomorphic modular forms indexed by $I$ and $J$. We say that $(f_m, g_n)_{m, n}$ is a \emph{modular grid} if the following conditions hold:
\begin{itemize}
	\item For each $m \in I$, we may write
	\[
	f_m(z) = R_m(q) + \sum\limits_{n \in J} a(m, n) q^n,
	\]
	where $R_m(q)$ is a Laurent polynomial in rational powers of $q$.
    \item For each $n \in J$, we may write
	\[
	g_n(z) = S_n(q) + \sum\limits_{m \in I} b(n, m) q^m,
	\]
	where $S_n(q)$ is a Laurent polynomial in rational powers of $q$.
	\item For each $m \in I$ and each $n \in J$, we have $a(m, n) = -b(n, m)$.
\end{itemize}
In some examples, the finite Laurent polynomials $R_m(q)$ and $S_n(q)$ are the principal parts of the weakly holomorphic modular forms, but this does not have to be the case.  We do not require that $I$ and $J$ consist of nonnegative indices in this general setting.

Thus, two sequences of weakly holomorphic modular forms constitute a modular grid if for almost all $m$ and $n$, the $n$th coefficient of the $m$th form in one sequence is the negative of the $m$th coefficient of the $n$th form in the other sequence. If $(f_m, g_n)_{m, n}$ is a nonzero modular grid indexed by $I$ and $J$, then $I$ and $J$ are necessarily infinite, since no nonconstant modular form has a finitely supported Fourier expansion.

In this section we construct row-reduced canonical bases for spaces of weakly holomorphic modular forms and use the modular pairing to prove Theorem~\ref{IntroThmDualitySimple}, thereby verifying that these bases form a modular grid. We begin by establishing row-reduced canonical bases $\parent{\Finf \nu k m}_{m \in \I \nu k}$ and $\parent{\Ginf \nu k m}_{m \in \J \nu k}$ for the space of holomorphic forms $M_k(\Gamma, \nu)$ and the space of cusp forms $\spcG k (\Gamma, \nu)$ respectively. The indexing sets $\I \nu k$ and $\J \nu k$ for these bases are then used to define indexing sets $\tI \nu k$ and $\tJ \nu k$ for $\finf \nu k m$ and $\ginf \nu k m$ respectively. We construct the principal parts of $\finf \nu k m$ and $\ginf \nu k m$, as elements of $\bbC((q))_{\widehat\Gamma, \nu}$, in terms of the coefficients of $\Finf \nu k m$ and $\Ginf \nu k m$. With these definitions in hand, we apply Theorem \ref{ThmFunkePairingFormExistence} together with some careful linear algebra to prove that these principal parts correspond to weakly holomorphic modular forms (Lemma \ref{Lemma: finf and ginf exist}). From here, it is easy to verify that these forms give canonical bases (Proposition \ref{Proposition: finf and ginf are canonical bases}). We conclude Section \ref{SecDualityWHMF} by proving Theorem \ref{IntroThmDualitySimple}.

\subsection{Construction of bases}\label{SubsecFormsatinf}

The spaces $M_k(\Gamma,\nu) $ and $\spcG k (\Gamma,\nu)$ have finite dimension, and so each has a basis of forms whose Fourier expansions at $\cusp \infty$ are in reduced echelon form. We denote these bases by
\begin{align*}
	\Finf \nu k m(z) &= q^m+\sum_{\substack{n \in \Znu \nu \\ m<n \not\in \I \nu k }} \Ainf \nu k mn q^n \ \ \textrm{for} \ \ m \in \I \nu k ,\\
	\Ginf \nu k n(z )&= q^n + \sum_{\substack{m \in \Znu \nu \\ n<m \not\in \J \nu k }} \Binf \nu k n m q^m \ \ \textrm{for} \ \ n \in \J \nu k .
\end{align*}
\noindent
Here, the finite sets of indices $\I \nu k \subseteq\Znu \nu _{\geq 0}$ and $\J \nu k \subseteq\Znu \nu _{> 0}$ are defined implicitly to be the indices of the reduced bases. Note that if $k<0,$ then $ \I \nu k$ and $\J\nu k$ are empty. Note also that the set of $\Ginf \nu k n$ may not be a subset of the set of $\Finf \nu k m$, if there are multiple cusps and thus multiple forms in the Eisenstein subspace of $M_k(\Gamma,\nu)$.

Define the indexing sets $\tI \nu k$ and $\tJ \nu k$ for $\parent{\finf \nu k m}_m$ and $\parent{\ginf \nu k n}_n$ by the disjoint unions
\begin{align*}
	\tI \nu k \sqcup \J {\conj \nu} {2-k}=\Znu {\conj \nu}_{>0}\sqcup \left(-\I \nu k \right), \quad \text{and }\quad
	\tJ \nu k \sqcup \I {\conj \nu} {2-k}=\Znu {\conj \nu}_{\geq 0}\sqcup \left(-\J \nu k \right).
\end{align*}
These indexing sets partition $\Znu \nu$ so that
\begin{align}
	\parent{-\tI { \nu} {k}} \sqcup \tJ {\conj\nu} {2-k} \ = \ \Znu \nu \ = \
	\tI {\conj\nu} {2-k} \sqcup \parent{-\tJ { \nu} {k}}. \label{Equation: IndexDef2}
\end{align}
While $\I \nu k$ and $\J \nu k$ index the bases $\Finf \nu k m$ and $\Ginf \nu k n$ for the spaces of holomorphic modular forms and cusp forms by the orders of their \emph{zeros} at $\infty$, the $\tI \nu k$ and $\tJ \nu k$ will be indexing sets for the orders of the \emph{poles} at $\infty$ of the canonical bases for spaces of weakly holomorphic forms.
We set $\finf \nu k m = \Finf \nu k {-m}$ for $-m\in \I \nu k$, and set $\ginf \nu k n = \Ginf \nu k {-n}$ for $-n\in \J \nu k$.

In order to define the remaining forms $\finf \nu k m$ and $\ginf \nu k n$, we begin by constructing their principal parts as elements of $\bbC((q))_{\widehat\Gamma, \nu}$. Theorem \ref{ThmFunkePairingFormExistence} will then guarantee the existence of the forms $\finf \nu k m$ and $\ginf \nu k n$ themselves. For $m \in -\Znu \nu$, let $\qinf \nu k m = \parent{\qinfpart \nu k m \indexmatrix}_\indexmatrix \in \C((q))_{\widehat\Gamma,\nu}$ be given by
\begin{align}
	\qinfpart \nu k m \indexmatrix (q) &= \begin{cases}
\nu(\errormatrix) \e^{-2\pi \i r m} q^{-m} & \ \text{if} \ \indexmatrix = \pm \errormatrix \smatrix 1r01 \ \text{with} \ \errormatrix \in \Gamma \ \text{and} \ r\in \Q, \\
	0 & \ \text{otherwise.}
	\end{cases} \nonumber \\
	\intertext{For $m \in \tI \nu k$ and $m>0$ we define $\principalfinf \nu k m = \qinf \nu k m - \finferror \nu k m$ where $\finferror \nu k m = \parent{\finferrorpart \nu k m \indexmatrix}_\indexmatrix \in \C((q))_{\widehat\Gamma,\nu}$ is given by}
	\finferrorpart \nu k m I (q) &= \sum\limits_{n \in \J {\conj \nu} {2-k}} \Binf {\conj \nu} {2-k} nm q^{-n}, \label{Defn: finferrorprincipal} \\
	\finferrorpart \nu k m \indexmatrix (q) &= \begin{cases}
	\nu(\errormatrix) \finferrorpart \nu k m I (\e^{2 \pi \i r} q) & \ \text{if} \ \indexmatrix = \pm \errormatrix \smatrix 1r01 \ \text{with} \ \mu \in \Gamma \ \text{and} \ r\in \Q, \\
	0 & \ \text{otherwise.}
	\end{cases} \nonumber \\
	\intertext{Similarly, for $n \in \tJ \nu k$ and $n\geq0$, we define $\principalginf \nu k n = \qinf \nu k n - \ginferror \nu k n$ where $\ginferror \nu k n = \parent{\ginferrorpart \nu k n \indexmatrix}_\indexmatrix \in \C((q))_{\widehat\Gamma,\nu}$ is given by}
	\ginferrorpart \nu k n I (q) &= \sum\limits_{m \in \I {\conj \nu} {2-k}} \Ainf {\conj \nu} {2-k} mn q^{-m}, \nonumber \\
	\ginferrorpart \nu k n \indexmatrix (q) &= \begin{cases}
	\nu(\errormatrix) \ginferrorpart \nu k n I (\e^{2 \pi \i r} q) & \ \text{if} \ \indexmatrix = \pm \errormatrix \smatrix 1r01 \ \text{with} \ \mu \in \Gamma \ \text{and} \ r\in \Q, \\
	0 & \ \text{otherwise.}
	\end{cases} \nonumber
\end{align}

We now consider the constant terms for $\finf \nu k m$ at the cusps of $\Gamma$ more closely. Let $\bbC_{\widehat\Gamma, \nu}$ be the subspace of $\bbC((q))_{\widehat\Gamma, \nu}$ whose components are constants.  Each element of this subspace represents a ``candidate'' for the collection of constant terms in the Fourier expansions of a modular form at each cusp. If $\gamma \in \widehat\Gamma$ and $0 \not\in \Zrho {\nu}{\gamma\infty}$, then $\finf \nu k m$ must vanish at $\cusp \gamma\infty$; likewise, every series in $\bbC((q))_{\widehat\Gamma, \nu}$ vanishes at $\gamma \infty$.

\begin{lemma}\label{Lemma: finf and ginf exist}
	Fix $m > 0$ in $\tI \nu k$ and $n \geq 0$ in $\tJ \nu k$. There exists $\mathbf{x} \in \bbC_{\widehat\Gamma, \nu}$ such that for each $h \in M_{2-k}(\Gamma, \conj \nu)$, we have \[
	\set{\principalfinf \nu k m + \mathbf x, h}_\Gamma = 0 \ \ \ \text{and} \ \ \
	\set{\principalginf \nu k n, h}_\Gamma = 0.
	\]
\end{lemma}

\begin{proof}
	We first prove the claim for $\principalfinf \nu k m$. As the modular pairing is bilinear, it suffices to prove the claim for $h$ in a spanning set for $M_{2-k}(\Gamma, \conj \nu)$.
	
We know that
\[
	M_{2-k}(\Gamma, \conj \nu) = S_{2-k}(\Gamma, \conj \nu)\oplus\spceis {2-k} \parent{\Gamma, {\conj\nu}},
\]
	 where the Eisenstein subspace $\spceis {2-k} \parent{\Gamma, {\conj\nu}}$ is the subspace of $ M_k(\Gamma, \nu)$ which is orthogonal to $S_k(\Gamma, \nu)$ with respect to the Petersson inner product. Whereas the forms in $S_k(\Gamma, \nu)$ vanish at every cusp, the Eisenstein component of a modular form in $M_k(\Gamma, \nu)$ captures the behavior of the form at the cusps, and is uniquely determined by the constant terms at each cusp. Using the modular pairing, it is straightforward to verify that
	\begin{equation*}
		\dim \bbC_{\widehat\Gamma, \nu} = \dim \spceis {k} \parent{\Gamma, {\nu}} +\dim \spceis {2-k} \parent{\Gamma, {\conj\nu}}
	\end{equation*}
	Therefore, if $\epsilon=\dim \spceis {2-k} \parent{\Gamma, {\conj\nu}}$ and $\parent{E_j}_{1\leq j \leq \epsilon}$ is any basis for $\spceis {2-k} \parent{\Gamma, {\conj\nu}}$, then there is at least one solution $\mathbf x\in \bbC_{\widehat\Gamma, \nu}$ to the system of linear equations given by
	\[
	 	 \{\mathbf x, E_j \}_\Gamma \,=\, -\set{\principalfinf \nu k m , E_j}_\Gamma \ \ 1\leq j \leq \epsilon.
	\]
	Take $\mathbf x$ to be any such solution.

	The collection $\parent{\Ginf {\conj \nu}{2-k}n }_{n \in \J {\conj \nu}{2-k}}$ is a basis for $S_{2-k}(\Gamma, \conj\nu)$. Observe that
	\[
		\set{\principalfinf \nu k m+ \mathbf{x}, \Ginf {\conj \nu}{2-k}n}_\Gamma = \set{\principalfinf \nu k m, \Ginf {\conj \nu}{2-k}n}_\Gamma,
	\]
	and the only contribution comes from the cusp $\cusp \infty$. By \eqref{Defn: finferrorprincipal}, we have
	\[
		\ainf \nu k m {-\ell} = - \Binf {\conj \nu} {2-k} \ell m
	\]
	for $\ell \in \J {\conj \nu} {2-k}$. Thus, we compute
	\begin{align*}
	\set{\principalfinf \nu k m + \mathbf{x}, \Ginf {\conj \nu}{2-k}n}_\Gamma &= 	\set{\principalfinf \nu k m, \Ginf {\conj \nu}{2-k}n}_\Gamma \\
	&=	\frac{\cuspwidth \Gamma \infty}{\widehat w} \parent{\Binf {\conj \nu}{2-k}nm - \Binf {\conj \nu}{2-k}nm - \sum\limits_{\ell \in \J {\conj \nu} {2-k}} \Binf {\conj \nu} {2-k} \ell m \Binf {\conj \nu} {2-k} n \ell} \\
	&= - \frac{\cuspwidth \Gamma \infty}{\widehat w} \parent{\sum\limits_{\ell \in \J {\conj \nu} {2-k}} \Binf {\conj \nu} {2-k} \ell m \Binf {\conj \nu} {2-k} n \ell} \\
	&= 0,
	\end{align*}
	where the last equality holds because $\Binf {\conj \nu} {2-k} n \ell = 0$ for $\ell \in \J {\conj \nu} {2-k}$.
Therefore the claim holds with $\mathbf x$ chosen as above.
	
	The proof for $\principalginf \nu k m$ is similar, except that we use $\parent{\Finf \nu k m}_{m \in \I {\conj \nu} {2-k}}$ as our spanning set for $M_{2-k}(\Gamma, \conj \nu),$ and do not need to separate out the Eisenstein component.
\end{proof}

Theorem \ref{ThmFunkePairingFormExistence} and Lemma \ref{Lemma: finf and ginf exist} together imply the existence of forms in $M_k^!(\Gamma, \nu)$ with principal parts given by $\principalfinf \nu k m + \mathbf x$ and $\principalginf \nu k m$. However, we must do a little more linear algebra to obtain the desired row-reduced bases for their respective spaces.

Let $\widetilde \pi : M_k^!(\Gamma, \nu) \to M_k^!(\Gamma, \nu)$ and $\widehat \pi : M_k^!(\Gamma, \nu) \to M_k^!(\Gamma, \nu)$ be defined as follows. Given any form $h(z) = \sum_{n\in \Znu \nu}a(n)q^n \in M^!_k(\Gamma,\nu)$, let $\widetilde \pi(h)$ and $\widehat \pi(h)$ denote the forms obtained by reducing $h$ against the forms $\Finf \nu k m$ and $\Ginf \nu k m$, so that
\begin{align*}
	\widetilde \pi(h)(z) \ & = \ h(z) - \sum_{m \in \I {\nu} {k}}a(m) \Finf \nu k m (z) \ = \ \sum_{\substack{n\in \Znu \nu\\ n\not \in\I {\nu} {k} }}\widetilde a(n)q^n,\\
	\widehat \pi (h)(z) \ &= \ h(z) -\sum_{n \in \J {\nu} {k}}a(n) \Ginf \nu k n (z)\ = \ \sum_{\substack{m \in \Znu \nu\\ m \not \in\J {\nu} {k} }}\widehat a(m)q^m.
\end{align*}
Informally, $\widetilde \pi$ row-reduces $h$ against the forms in $M_k(\Gamma, \nu)$, and $\widehat \pi$ row-reduces $h$ against the forms in $\spcG k(\Gamma, \nu)$.

By Lemma \ref{Lemma: finf and ginf exist} and Theorem \ref{ThmFunkePairingFormExistence}, for each $m \in \tI \nu k$ with $m > 0$ there exists at least one form $f \in M_k^!(\Gamma, \nu)$ such that for each $\gamma \in \widehat{\Gamma}$, we have $f^\gamma = \principalfinfpart \nu k m \gamma + O(1)$. Let $\finf \nu k m$ be the unique such form with
\begin{align*}
	\finf \nu k m &= \widetilde \pi\parent{\finf \nu k m}.
\end{align*}
Similarly, for each $n \in \tJ \nu k$ with $n \geq 0$ there exists at least one form $g \in M_k^!(\Gamma, \nu)$ such that for each $\gamma \in \widehat{\Gamma}$, we have $g^\gamma = \principalginfpart \nu k n \gamma + o(1)$. Let $\ginf \nu k n$ be the unique such form with
\begin{align*}
\ginf \nu k n &= \widehat{\pi}\parent{\ginf \nu k n}.
\end{align*}

Our next proposition confirms that the forms we have constructed actually comprise row-reduced bases for $\spcfinf k (\Gamma, \nu)$ and its subspace $\spcginf k \parent{\Gamma, \nu}$ of forms which vanish at every cusp other than $\infty$.

\begin{prop}\label{Proposition: finf and ginf are canonical bases}

	Assume the notation above. The sequence $\parent{\finf \nu k m}_{m \in \tI \nu k}$ is a basis for $\spcfinf k \parent{\Gamma, \nu}$, and the sequence $\parent{\ginf \nu k n}_{n \in \tJ \nu k}$ is a basis for $\spcginf k \parent{\Gamma, \nu}$. Furthermore, if $f(z) = \sum\limits_{n \in \Znu \nu} a(n) q^n \in \spcfinf k (\Gamma, \nu)$ and $g(z) = \sum\limits_{m \in \Znu \nu} b(m) q^m \in \spcginf k (\Gamma, \nu),$
	then
\begin{equation}\label{Equation: BasisExp}
	f = \sum\limits_{m \in \tI \nu k} a(-m) \finf \nu k {m} \ \ \ \text{and} \ \ \
	g = \sum\limits_{n \in \tJ \nu k} b(-n) \ginf \nu k {n}.
\end{equation}
\end{prop}

\begin{proof}
	By construction, $\finf \nu k m |_{k, \nu} \gamma = O(1)$ and $\ginf \nu k n |_{k, \nu} \gamma = o(1)$ for all $\gamma \in \widehat \Gamma$ with $\cusp{\gamma \infty} \neq \cusp{\infty}$. Therefore $\finf \nu k m \in \spcfinf k \parent{\Gamma, \nu}$ and $\ginf \nu k n \in \spcginf k \parent{\Gamma, \nu}$ as desired. Moreover, the sequences $\parent{\finf \nu k m}_{m}$ and $\parent{\ginf \nu k n}_{n}$ are each linearly independent over $\bbC$, since their terms have distinct orders at infinity.
	
We will show that \eqref{Equation: BasisExp} holds for $\parent{\finf \nu k m}_{m}$. This establishes that the set spans $\spcfinf k \parent{\Gamma, \nu}$, and is therefore a basis. The argument for $\parent{\ginf \nu k n}_{n}$ is completely analogous.
	
Given $f(z)$ as above, consider the function $h(z)$ obtained by taking the difference
	\[
		h(z) = f(z) - \sum\limits_{m \in \tI \nu k} a(-m) \finf \nu k {m} (z) = \sum\limits_{n \in \tJ {\conj\nu} {2-k}} c(n) q^n,
	\]
which has finitely many terms because both $\tI \nu k$ and the indices of the Fourier coefficients of $f(z)$ are bounded below. Here we have used the fact that $\Znu \nu=\parent{-\tI { \nu} {k}} \sqcup \tJ {\conj\nu} {2-k}.$ If $n \in \tJ {\conj \nu}{2-k}$, then by Theorem \ref{ThmPairing} we know that
$\set{h, \ginf {\conj \nu}{2-k}n}_\Gamma = 0.$
Only the contribution from the cusp $\cusp \infty$ may be non-trivial since $\ginf {\conj \nu}{2-k}n$ vanishes at the other cusps, so
\begin{align*}
	\set{h, \ginf {\conj \nu}{2-k}n}_\Gamma = \frac{\cuspwidth \Gamma \infty}{\widehat w} \parent{c(n) + \sum\limits_{\ell \in \Znu \nu} \binf {\conj \nu}{2-k}n{-\ell} c(\ell)} = 0.
	\end{align*} But $c(\ell) = 0$ for all $\ell \in -\tI \nu k$, and $\binf {\conj \nu}{2-k}n{-\ell} = 0$ for all $\ell \in \tJ {\conj \nu}{2-k},$ so
	\[
		 \set{h, \ginf {\conj \nu}{2-k}n}_\Gamma = \frac{\cuspwidth \Gamma \infty}{\widehat w} c(n) =0.
	\]
	Since $n \in \tJ {\conj \nu}{2-k}$ was arbitrary, $h$ is identically zero, as desired.
\end{proof}

As a corollary, the bases $\parent{\finf \nu k m}_m$ and $\parent{\ginf \nu k n}_n$ are independent of our choice of $\widehat{\Gamma}$.

We are now in a position to prove Theorem \ref{IntroThmDualitySimple}. We restate it here for ease of reference.   Recall that the coefficients $\ainf { \nu} k mn$ and $\binf {\conj \nu} {2-k} nm$ are defined from the Fourier expansions
\begin{align*}
	\finf {\nu} {k} m(z) &= q^{-m} + \sum_{\substack {n \in \Znu \nu \\ -n \not\in \tI \nu k}} \ainf {\nu} {k} mn q^n \ \ \textrm{for} \ \ m \in \tI {\nu} {k}, \\
	\ginf \nu k n (z) &= q^{-n} + \sum_{\substack {m \in \Znu \nu \\ -m \not\in \tJ \nu k}} \binf \nu k nm q^m \ \ \textrm{for} \ \ n \in \tJ \nu k.
\end{align*}

\begin{theorem}\label{ThmDualityDetail}
Let $\Gamma $ be a commensurable subgroup, let $k\in \frac 1 2 \Z,$ and let $\nu$ be any consistent finite weight $k$ multiplier. The coefficients of the forms $\parent{\finf \nu k m(z)}_m$ and $\parent{\ginf {\conj \nu} {2-k} n(z)}_n$ satisfy
\[
	\ainf { \nu} k mn = - \binf {\conj \nu} {2-k} nm.
\]
\end{theorem}

\begin{proof}
We apply the modular pairing to $\finf \nu k m$ and $\ginf {\conj \nu} {2-k} n$. By Lemma \ref{Lemma: finf and ginf exist}, these are both weakly holomorphic functions, and so by Theorem \ref{ThmPairing} we see \[
\set{\finf \nu k m, \ginf {\conj \nu} {2-k} n}_\Gamma = 0.
\] On the other hand, since $\finf \nu k m \ginf {\conj \nu} {2-k} n$ vanishes at every cusp other than $\cusp \infty$, we have by \eqref{Equation: Formal modular pairing} that
\begin{align}\label{pairingeq}
\set{\finf \nu k m, \ginf {\conj \nu} {2-k} n}_\Gamma &= \frac{\cuspwidth \Gamma \infty}{\widehat w} \parent{\ainf \nu k m n + \binf {\conj \nu} {2-k} n m + \sum\limits_{\ell \in \Znu \nu} \ainf \nu k m {\ell} \binf {\conj \nu} {2-k} n {-\ell}}.
\end{align}
Note $\ainf \nu k m {\ell} = 0$ if $\ell \in -\tI {\nu} k$, and $\binf {\conj \nu} {2-k} n {-\ell} = 0$ if $\ell \in \tJ {\conj \nu} {2-k}.$ But by \eqref{Equation: IndexDef2} this covers every $\ell$, so \eqref{pairingeq} reduces to
\[
\set{\finf \nu k m, \ginf {\conj \nu} {2-k} n}_\Gamma = \frac{\cuspwidth \Gamma \infty}{\widehat w} \parent{\ainf \nu {2-k} m n + \binf {\conj \nu} k n m} = 0.
\]

This concludes the proof.
\end{proof}

In particular, $(\finf \nu k m, \ginf {\conj \nu} {2-k} n)_{m, n}$ constitutes a modular grid.

\section{Generalizations}\label{SecGeneralizations}
In this section we consider generalizations of Theorem~\ref{IntroThmDualitySimple} to spaces with different vanishing conditions and forms with poles at other cusps.

\subsection{Vanishing conditions at cusps}\label{SubsecVanishingForms}

The proof of coefficient duality above requires that the product $\finf \nu {2-k} m \ginf {\conj \nu} k n$ vanishes at each cusp other than $\cusp \infty$. However, the requirement that $\ginf {\conj \nu} k n$ vanishes at each cusp other than $\cusp \infty$ is merely a technical convenience, and duality theorems hold for row-reduced bases for other subspaces of $\spcfinf k \parent{\Gamma, \nu}$.

One such example is given in~\cite{Haddock-Jenkins}. Consider the space of weakly holomorphic modular forms for $\Gamma_0(4)$ with poles only at the cusp at $\infty$, and the subspaces consisting (respectively) of forms which vanish at the cusp at $0$ and forms which vanish at the cusp at $1/2$. Defining row-reduced canonical bases ${h_{k, m}^{(4)}}$ and $i_{k, m}^{(4)}$ for these subspaces, the product of $h_{k, m}^{(4)}$ and $i_{k, n}^{(4)}$ will be a modular form of weight $2$ vanishing at all cusps other than $\infty$. Since all such forms have zero constant term, it follows that the Fourier coefficients of the $h_{k, m}^{(4)}$ are dual to the Fourier coefficients of the $i_{k, m}^{(4)}$.

To generalize this, let $U \subseteq \cusps \Gamma$ be any set of cusps, and write $\overline U$ for the complement of $U$ in $\cusps \Gamma$. Thus we have $U \sqcup \overline{U} = \cusps \Gamma.$ We define $M_k(\Gamma, \nu, U)$ to be the space of weight $k$ holomorphic modular forms $f$ for the group $\Gamma$ and multiplier $\nu$ such that $f$ vanishes at each cusp in $U$. The space $M_k(\Gamma, \nu, U)$ has finite dimension, and so has a basis of forms whose Fourier coefficients at $\cusp \infty$ are row-reduced. We denote this basis by
\[
	\Fvan \nu k m U (z) = q^m+\sum_{\substack{n\in \Znu \nu _{\geq 0}\\ n \not\in \Ivan \nu k U}} \Avan \nu k mn U q^n \ \ : \ \ m\in \Ivan \nu k U .
\]
Here, the finite set $\Ivan \nu k U \subseteq \Znu \nu _{\geq 0}$ is defined implicitly to be the set of indices for the reduced basis. Writing $\Omega=\cusps \Gamma$, we see that
\[
	\parent{\Fvan \nu k m \emptyset}_{m} \ \text{and} \ \parent{\Fvan \nu k n {\Omega
	}}_{n}
\]
furnish bases for $M_k(\Gamma)$ and $S_k(\Gamma)$. Indeed, $\parent{\Fvan \nu k m \emptyset}_{m} = \parent{\Finf \nu k m}_{m}$ and $\parent{\Fvan \nu k n {\Omega
}}_{n} = \parent{\Ginf \nu k n}_{n}$.

Let $\spcfinf k(\Gamma, \nu, U)$ denote the space of weakly holomorphic modular forms which are holomorphic at every cusp other than $\cusp \infty$, and which vanish at each cusp in $U$ other than $\infty$. We see that
\[
	\spcfinf k(\Gamma, \nu, \emptyset) = \spcfinf k(\Gamma, \nu) \ \ \ \text{and} \ \ \ \spcfinf k(\Gamma, \nu, \Omega
	) = \spcginf k(\Gamma, \nu).
\]

Finally, define
\[
	\spceis k (\Gamma, \nu, U) = \spceis k (\Gamma, \nu) \cap M_k(\Gamma, \nu, U)
\]
and
\[
	\bbC_{\widehat\Gamma, \nu}(U) = \set{\textbf{x} = \parent{x^\lambda}_\lambda \in \bbC_{\widehat\Gamma, \nu} \ \ : \ \ x^\lambda = 0 \ \text{whenever} \ \lambda \infty \in U}.
\]
It is straightforward to verify that
\begin{align*}
		M_{2-k}(\Gamma, \conj \nu) &= M_{2-k}(\Gamma, \conj \nu, \overline U) \oplus \spceis {2-k} \parent{\Gamma, {\conj\nu}, U}
		\intertext{and that}
		\dim \bbC_{\widehat\Gamma, \nu}(U) &= \dim \spceis {k} \parent{\Gamma, {\nu}, U} + \dim \spceis {2-k} \parent{\Gamma, {\conj\nu}, U}.
\end{align*}

With these tools in hand, we define the row-reduced canonical basis
\[
	\fvan \nu k m U(z) = q^{-m} + \sum\limits_{\substack{n \in \Znu \nu \\ -n \not\in \tIvan \nu k U}} \avan {\nu} {k} mn U q^n \ \ : \ \ m \in \tIvan \nu k U .
\]
for $\spcfinf k(\Gamma, \nu, U)$ in a manner totally analogous to the construction of $\finf \nu k m$ and $\ginf \nu k m$ in Subsection \ref{SubsecFormsatinf}.

\begin{prop}\label{Proposition: van-bases are canonical}
	Assume the notation above. The sequence $\parent{\fvan \nu k m U}_{m \in \tIvan \nu k U}$ is a basis for $\spcfinf k \parent{\Gamma, \nu, U}$. Furthermore, if $f(z) = \sum\limits_{n \in \Znu \nu} a(n) q^n \in \spcfinf k (\Gamma, \nu, U)$, then
\[
	f = \sum\limits_{m \in \tIvan \nu k U} a(-m) \fvan \nu k {m} U.
\]
\end{prop}

As a corollary, if $V = U \cup \set \infty$ then $\tIvan \nu k U = \tIvan \nu k V$ and $\fvan \nu k m U = \fvan \nu k m V$ for $m \in \tIvan \nu k U$. This is unsurprising, since by definition $\spcfinf k (\Gamma, \nu, U) = \spcfinf k (\Gamma, \nu, V)$. Moreover, as both $\cusps \Gamma$ and $\spcfinf k (\Gamma, \nu, U)$ are independent of our choice of $\widehat{\Gamma}$, we see that $\parent{\fvan \nu k m U}_m$ is independent of $\widehat{\Gamma}$.

Now, in analogy with Theorem \ref{IntroThmDualitySimple}/Theorem~\ref{ThmDualityDetail}, we have the following theorem.
\begin{theorem}\label{ThmVanishingDuality}
	Let $\Gamma$ be a commensurable subgroup, let $k\in \frac 1 2 \Z,$ let $\nu$ be any consistent finite weight $k$ multiplier, and let $U$ be a set of cusps for $\Gamma$. The coefficients of the forms $\parent{\fvan \nu k m U}_m$ and $\parent{\fvan {\conj \nu} {2-k} n {\overline U}}_n$ satisfy
	\[\avan { \nu} k mn U = - \avan {\conj \nu} {2-k} nm {\overline U}.\]
\end{theorem}

In particular, $\parent{\fvan \nu k m U, \fvan {\conj \nu} {2-k} n {\conj U}}_{m, n}$ constitutes a modular grid.

\subsection{Duality at other cusps}\label{SubsecOtherCusps}

In Section \ref{SecDualityWHMF} we demonstrated that the forms $\parent{\finf \nu k m}_{m}$ are dual to the forms $\parent{\ginf {\conj \nu} {2-k} n}_{n}$. Now fix $\gamma \in \widehat \Gamma$. Recall that $\Gamma^\gamma = \gamma\inv \Gamma \gamma$ and $\nu^\gamma(\mu) = \nu(\gamma \mu \gamma\inv)$.  It is natural to ask whether the forms $\parent{\finf \nu k m |_k \gamma}_{m} \subseteq M_k^!(\Gamma^\gamma, \nu^\gamma)$ and $\parent{\ginf {\conj \nu} {2-k} n |_k \gamma}_{n} \subseteq M_k^!(\Gamma^\gamma, \nu^\gamma)$ still exhibit coefficient duality with each other. Unsurprisingly, the answer is no: these forms have poles only at $\gamma\inv \infty$, and exhibit coefficient duality at the cusp $\gamma\inv \infty$, but generally fail to exhibit duality \textit{at $\infty$}. However, we may ask the following:

\begin{question}\label{q1}
Are there sets of weakly holomorphic forms that are dual to the sets $\parent{\finf \nu k m |_k \gamma}_{m}$ and $\parent{\ginf {\conj \nu} {2-k} n |_k \gamma}_{n}$?
\end{question}

Alternatively, we could ask whether a basis for the space of forms which are modular for $\Gamma$ with poles only at the cusp $\gamma \infty$ exhibits coefficient duality with any other sequence of weakly holomorphic modular forms. A basis for this space is given by $\parent{\finf {\nu^\gamma} k m |_k \gamma\inv}_{m}$ indexed by $\tI {\nu^\gamma} k$. If we add the condition that the forms vanish at every other cusp, a basis is given by $\parent{\ginf {\nu^\gamma} k n |_k \gamma\inv}_{n}$ indexed by $\tJ {\nu^\gamma} k$. Thus, the question of coefficient duality for these spaces is essentially the same as Question \ref{q1}, with $\gamma$ replaced by $\gamma\inv$ to avoid triple superscripts.

The answer to Question \ref{q1} is yes: intuitively, a sequence of forms with poles at $\gamma \infty$ should be dual to a sequence of forms with poles at $\gamma\inv \infty$, but the fact that we are expanding at $\infty$ complicates this duality somewhat. In this subsection we construct the forms dual to those described in Question \ref{q1}. In fact, we will prove a slightly more general result (see Theorem \ref{ThmDualityPolesAnywhere}).

For $\gamma \in \widehat{\Gamma}$, define
\begin{align*}
	\Ialpha \nu k \gamma = \I {\nu^{\gamma}} k \ \ \ \ \textrm{ and } \ \ \ \ \Jalpha \nu k \gamma = \J {\nu^{\gamma}} k.
\end{align*}
Similarly, define
\begin{align*}
	\Fgamma \nu k m \gamma = \Finf {\nu^{\gamma}} k m |_{k} \gamma\inv \ \ \ \ \textrm{ and } \ \ \ \ \Ggamma \nu k n \gamma = \Ginf {\nu^{\gamma}} k n |_k \gamma\inv.
\end{align*}
By construction, the collection of forms $\parent{\Fgamma \nu k m \gamma}_{m \in \Ialpha \nu k \gamma}$ is a basis for $M_k(\Gamma, \nu)$ which is row-reduced at $\cusp {\gamma \infty}$. Likewise, the collection of forms $\parent{\Ggamma \nu k n \gamma}_{n \in \Jalpha \nu k \gamma}$ is a basis for $\spcG k (\Gamma, \nu)$ which is row-reduced at $\cusp {\gamma \infty}$.

For $m \in \Ialpha \nu k \gamma$ we write
\begin{align*}
	\Fgamma \nu k m \gamma (z) &= \sum\limits_{n \in \Znu \nu _{\geq 0}} \Agamma \nu k m n \gamma q^n + \begin{cases}
		\nu(\errormatrix) \e^{2 \pi \i r m} q^m & \ \text{if} \ \gamma = \pm \errormatrix \smatrix 1r01 \ \text{with} \ \errormatrix \in \Gamma \ \text{and} \ r\in \Q, \\
		0 & \ \text{otherwise.}
	\end{cases}
	\intertext{Likewise, for $n \in \Jalpha \nu k \gamma$ we write }
	\Ggamma \nu k n \gamma (z) &= \sum\limits_{m \in \Znu \nu _{> 0}} \Bgamma \nu k nm \gamma q^m + \begin{cases}
		\nu(\errormatrix) \e^{2 \pi \i r n} q^n & \ \text{if} \ \gamma = \pm \errormatrix \smatrix 1r01  \ \text{with} \ \errormatrix \in \Gamma \ \text{and} \ r\in \Q, \\
		0 & \text{ otherwise.}
	\end{cases}
\end{align*}

\begin{proposition}\label{FGcusps}
Assume the notation above. Let $\gamma, \gamma\prm \in \widehat{\Gamma}$ be any two matrices with $\cusp{\gamma \infty} = \cusp {\gamma\prm \infty}$. Then
\[\Fgamma \nu k m \gamma=\nu(\mu)\e^{2\pi\i r m}\Fgamma \nu k m {\gamma\prm} \ \ \text{and} \ \
\Ggamma \nu k n \gamma=\nu(\mu)\e^{2\pi\i r n}\Ggamma \nu k n {\gamma\prm}\]
where $\gamma\prm = \pm \errormatrix \gamma \smatrix 1r01$ with $\errormatrix \in \Gamma$ and $r \in \Q$.
\end{proposition}
\begin{proof}
We compute
\begin{eqnarray*}
\Fgamma \nu k m \gamma &=& \Fgamma \nu k m \gamma|_k\gamma\prm {\gamma\prm}\inv\\
&=&\nu(\errormatrix) \Finf {\nu^{\gamma}} k m|_k\smatrix 1r01 {\gamma\prm}\inv\\
&=&\nu(\errormatrix) \e^{2\pi\i r m}\Finf {\nu^{\gamma\prm}} k m|_k {\gamma\prm}\inv\\
&=&\nu(\errormatrix) \e^{2\pi\i r m}\Fgamma \nu k m {\gamma\prm}.
\end{eqnarray*}
The argument for $\Ggamma \nu k n \gamma$ is similar.
\end{proof}
We now construct sequences of weakly holomorphic modular forms that are indexed by poles at one cusp, with row-reduction at a (possibly) different cusp.
In analogy with $\Ialpha \nu k \gamma$ and $\Jalpha \nu k \gamma$ above, we define
\begin{align*}
	\tIalpha \nu k \alpha \gamma &= \begin{cases}
	\tI {\nu^\alpha} k & \text{if } \cusp {\alpha \infty} = \cusp {\gamma \infty}, \\
	\Zrho {\conj \nu} {\cusp{\alpha \infty}, > 0} & \text{if } \cusp {\alpha \infty} \neq \cusp {\gamma \infty},
	\end{cases} \\
	\tJalpha \nu k \alpha \gamma &= \begin{cases}
	\tJ {\nu^{\alpha}} k & \text{if } \cusp {\alpha \infty} = \cusp {\gamma \infty}, \\
	\Zrho {\conj \nu} {\cusp{\alpha \infty}, \geq 0} & \text{if } \cusp {\alpha \infty} \neq \cusp {\gamma \infty}.
	\end{cases}
\end{align*}
Note that for all $\alpha, \gamma \in \widehat{\Gamma}$, we have
\begin{align*}
	\parent{-\tIalpha { \nu} {k} \alpha \gamma} \sqcup \tJalpha {\conj\nu} {2-k} \alpha \gamma =
	\Zrho \nu {\cusp{\alpha \infty}} =
	\tIalpha {\conj\nu} {2-k} \alpha \gamma \sqcup \parent{-\tJalpha { \nu} {k} \alpha \gamma}.
\end{align*}

For $m \in -\Zrho \nu {\cusp {\alpha \infty}}$, let $\qalpha \nu k m \alpha = \parent{\qalphapart \nu k m \alpha \lambda}_\lambda \in \C((q))_{\widehat\Gamma,\nu}$ be given by
\begin{align}
\qalphapart \nu k m \alpha \indexmatrix (q) &= \begin{cases}
\nu(\errormatrix) \e^{-2\pi \i r m} q^{-m} & \ \text{if} \ \indexmatrix = \pm \errormatrix \alpha \smatrix 1r01  \ \text{with} \ \errormatrix \in \Gamma \ \text{and} \ r\in \Q, \\
0 & \text{ otherwise.}
\end{cases} \nonumber \\
\intertext{For $m > 0$ in $\tIalpha \nu k \alpha \gamma$, we define $\principalfalpha \nu k m \alpha \gamma = \qalpha \nu k m \alpha - \falphaerror \nu k m \alpha \gamma$ where $\falphaerror \nu k m \alpha \gamma = \parent{\falphaerrorpart \nu k m \alpha \gamma \lambda}_\lambda \in \C((q))_{\widehat\Gamma,\nu}$ is given by}
\falphaerrorpart \nu k m \alpha \gamma \gamma (q) &= \frac{\cuspwidth \Gamma {\cusp {\alpha \infty}}}{\cuspwidth \Gamma {\gamma \infty}} \sum\limits_{n \in \Jalpha {\conj \nu} {2-k} {\cusp {\gamma \infty}}} \Bgamma {\nu^\alpha} k n m {\alpha\inv \gamma} q^{-n}, \label{Defn: fmu} \\
\falphaerrorpart \nu k m \alpha \gamma \indexmatrix (q) &= \begin{cases}
\nu(\errormatrix) \falphaerrorpart \nu k m \alpha \gamma \gamma (\e^{2 \pi \i r} q) & \ \text{if} \ \indexmatrix = \pm \errormatrix \gamma \smatrix 1r01  \ \text{with} \ \errormatrix \in \Gamma \ \text{and} \ r\in \Q, \\
0 & \ \text{otherwise.}
\end{cases} \nonumber
 \\
\intertext{For $n \geq 0$ in $\tJalpha \nu k \alpha \gamma$, we define $\principalgalpha \nu k n \alpha \gamma = \qalpha \nu k n \alpha - \galphaerror \nu k n \alpha \gamma$ where $\galphaerror \nu k n \alpha \gamma = \parent{\galphaerrorpart \nu k n \alpha \gamma \lambda}_\lambda \in \C((q))_{\widehat\Gamma,\nu}$ is given by}
\galphaerrorpart \nu k n \alpha \gamma \gamma &= \frac{\cuspwidth \Gamma {\cusp {\alpha \infty}}}{\cuspwidth \Gamma {\gamma \infty}} \sum\limits_{m \in \Ialpha {\conj \nu} {2-k} {\cusp {\gamma \infty}}} \Agamma {\nu^\alpha} k m n {\alpha\inv \gamma} q^{-m}, \nonumber \\
\galphaerrorpart \nu k n \alpha \gamma \indexmatrix &= \begin{cases}
\nu(\errormatrix) \galphaerrorpart \nu k n \alpha \gamma \gamma (\e^{2 \pi \i r} q) & \ \text{if} \ \lambda = \pm \errormatrix \gamma \smatrix 1r01 \ \text{with} \ \errormatrix \in \Gamma \ \text{and} \ r\in \Q, \\
0 & \ \text{otherwise.}
\end{cases} \nonumber
\end{align}

\begin{lemma}\label{Lemma: fmu and gmu exist}
	Fix $m > 0$ in $\tIalpha \nu k \alpha \gamma$ and $n \geq 0$ in $\tJalpha \nu k \alpha \gamma$. There exists $\mathbf{x} \in \C((q))_{\widehat\Gamma,\nu}$ such that for each $h \in M_{2-k}(\Gamma, \conj \nu)$, we have
	\begin{align*}
		\set{\principalfalpha \nu k m \alpha \gamma + \mathbf{x}, h}_\Gamma &= 0 \ \text{and} \\
		\set{\principalgalpha \nu k n \alpha \gamma, h}_\Gamma &= 0.
	\end{align*}
\end{lemma}

\begin{proof}
	We proceed by analogy with Lemma \ref{Lemma: finf and ginf exist}. We first prove the claim for $\principalfalpha \nu k m \alpha \gamma$. It suffices to prove the claim for $h$ in a spanning set for $M_{2-k}(\Gamma, \conj \nu)$.
	
	The existence of $\mathbf x$ so that
	\[\set{ \mathbf{x}, E}_\Gamma = -\set{\principalfalpha \nu k m \alpha \gamma ,E}_\Gamma, \]
	for each Eisenstein series $E\in \spceis {2-k}(\Gamma ,\conj\nu)$ follows exactly as in the proof of Lemma \ref{Lemma: finf and ginf exist}.
		
	 The set $\parent{\Ggamma {\conj \nu}{2-k}n \gamma }_{n \in \J {\conj \nu}{2-k}}$ is a basis for $S_{2-k}(\Gamma, \conj\nu)$. Observe that
	\[
		\set{\principalfalpha \nu k m \alpha \gamma+ \mathbf{x}, \Ggamma {\conj \nu}{2-k}n \gamma}_\Gamma = \set{\principalfalpha \nu k m \alpha \gamma,\Ggamma {\conj \nu}{2-k}n \gamma}_\Gamma
	\]
	and the only contribution can come from the cusps $\cusp {\alpha\infty}$ and $\cusp {\gamma\infty}$. By \eqref{Defn: fmu}, we have
	\[
		\aalpha {\nu^\gamma} k m {-\ell} {\gamma\inv \alpha} I = - \frac{\cuspwidth \Gamma {\alpha \infty}}{\cuspwidth \Gamma {\gamma \infty}} \Bgamma {\conj \nu^\alpha} {2-k} \ell m {\alpha\inv \gamma}
	\] for $\ell \in \Jalpha {\conj \nu} {2-k} \gamma$. Thus, for any $\mathbf{x} \in \bbC_{\widehat\Gamma,\nu}$ we compute
	\begin{align*}
		\set{\principalfalpha \nu k m \alpha \gamma + \mathbf{x}, \Ggamma {\conj \nu}{2-k}n \gamma}_\Gamma &=
				\set{\principalfalpha \nu k m \alpha \gamma, \Ggamma {\conj \nu}{2-k}n \gamma}_\Gamma\\
		&= \frac{\cuspwidth \Gamma {\alpha \infty}}{\widehat w}  \Bgamma {\conj \nu^\alpha} {2-k} nm {\alpha\inv \gamma} + \frac{\cuspwidth \Gamma {\gamma \infty}}{\widehat w} \parent{- \frac{\cuspwidth \Gamma {\alpha \infty}}{\cuspwidth \Gamma {\gamma \infty}} \Bgamma {\conj \nu^\alpha} {2-k} nm {\alpha\inv \gamma}} \\
		& \ \ \ \ \ \ \ - \frac{\cuspwidth \Gamma {\gamma \infty}}{\widehat w} \parent{\frac{\cuspwidth \Gamma {\alpha \infty}}{\cuspwidth \Gamma {\gamma \infty}} \sum\limits_{\ell \in \Jalpha {\conj \nu} {2-k} \gamma} \Bgamma {\conj \nu^\alpha} {2-k} \ell m {\alpha\inv \gamma} \Bgamma {\conj \nu^\gamma} {2-k} n {\ell} I} \\
		&=- \frac{\cuspwidth \Gamma {\alpha \infty}}{\widehat w}  \parent{\sum\limits_{\ell \in \Jalpha {\conj \nu} {2-k} \gamma} \Bgamma {\conj \nu^\alpha} {2-k} \ell m {\alpha\inv \gamma} \Bgamma {\conj \nu^\gamma} {2-k} n {\ell} I} \\
		&= 0,
	\end{align*} where the last equality holds because $\Bgamma {\conj \nu^\gamma} {2-k} n {\ell} I = 0$ for $\ell \in \Jalpha {\conj \nu} {2-k} \gamma$.  Therefore the claim holds with $\mathbf x$ chosen as above.
	
	The proof for $\principalgalpha \nu k m \alpha \gamma$ is similar, except that we use $\parent{\Fgamma \nu k m \gamma}_{m \in \I {\conj \nu} {2-k}}$ as our spanning set for $M_{2-k}(\Gamma, \conj \nu),$ and we do not need to separate out the Eisenstein component.
\end{proof}

We now construct projection maps $\widetilde \pi_{\gamma} : M_k^!(\Gamma, \nu) \to M_k^!(\Gamma, \nu)$ and $\widehat \pi_{\gamma} : M_k^!(\Gamma, \nu) \to M_k^!(\Gamma, \nu)$ as follows. Given any form $h \in M^!_k(\Gamma,\nu)$, for each $\lambda \in \widehat{\Gamma}$, write $h^\lambda(z) = \sum\limits_{n \in \Zrho \nu {\cusp {\lambda \infty}}} a^\lambda(n) q^n$.  Then define $\widetilde \pi_{\gamma}(h)$ and $\widehat \pi_{\gamma}(h)$ to be the forms obtained by reducing $h$ against the forms $\Fgamma \nu k m {\gamma}(z)$ and $\Ggamma \nu k m {\gamma} (z)$ at the cusp $\cusp {\gamma \infty}$, so that
\begin{align*}
\widetilde \pi_{\gamma} (h) \ &= \ h - \sum_{m \in \Ialpha {\nu} {k} {\gamma}} a^{\gamma}(m) \Fgamma \nu k m {\gamma}, \ \text{and} \\
\widehat \pi_{\gamma} (h) \ &= \ h - \sum_{n \in \Jalpha {\nu} {k} {\gamma}}a^{\gamma}(n) \Ggamma \nu k n {\gamma}.
\end{align*}
In particular, $\widetilde \pi_I = \widetilde \pi$ and $\widehat \pi_I = \widehat \pi$.

Observe that if $f \in M_k(\Gamma, \nu)$ then for any $\gamma \in \widehat{\Gamma}$, we have $\widetilde{\pi}_\gamma(f) = 0$. Likewise if $g \in \spcG k(\Gamma, \nu),$ then for any $\gamma \in \widehat{\Gamma}$, we have $\widehat{\pi}_\gamma(g) = 0$.

By Lemma \ref{Lemma: fmu and gmu exist} and Theorem \ref{ThmFunkePairingFormExistence}, for each $m > 0$ in $\tIalpha \nu k\alpha \gamma$ there exists at least one form $f \in M_k^!(\Gamma, \nu)$ so that $f|_k\lambda = \principalfalphapart \nu k m \alpha \gamma \lambda + O(1)$ for each $\lambda\in \widehat\Gamma$. Let $\falpha \nu k m \alpha \gamma$ be the unique such form with
\begin{align*}
	\falpha \nu k m \alpha \gamma &= \widetilde \pi_\gamma\parent{\falpha \nu k m \alpha \gamma}.
\end{align*}
If $\cusp{\alpha \infty} = \cusp{\gamma \infty}$, then for $m \leq 0$ in $\tIalpha \nu k\alpha \gamma$ we set $\falpha \nu k m \alpha \gamma = \Fgamma \nu k {-m} \alpha$.

Similarly, for each $n \geq 0$ in $\tJalpha \nu k\alpha \gamma$, there exists at least one form $g \in M_k^!(\Gamma, \nu)$ such that $g|k\lambda =\principalgalphapart \nu k n \alpha \gamma \lambda+o(1)$ for each $\lambda\in \widehat \Gamma$. Let $\galpha \nu k n \alpha \gamma$ be the unique such form with
\begin{align*}
\galpha \nu k n \alpha \gamma &= \widehat{\pi}_\gamma \parent{\galpha \nu k n \alpha \gamma}.
\end{align*}
If $\cusp{\alpha \infty} = \cusp{\gamma \infty}$, then for $n < 0$ in $\tJalpha \nu k\alpha \gamma$ we set $\galpha \nu k n \alpha \gamma = \Ggamma \nu k {-n} \alpha$.

The following proposition shows that the choice of a matrix $\gamma$ with $\gamma \infty = \rho$ does not affect the row reduction at the cusp $\rho$.

\begin{prop}\label{Proposition: Second matrix determines cusp}
	Assume the notation above. Let $\gamma, \gamma\prm \in \widehat{\Gamma}$ be any two matrices with $\cusp{\gamma \infty} = \cusp {\gamma\prm \infty}$. Then \[
	\widetilde{\pi}_\gamma = \widetilde{\pi}_{\gamma\prm} \ \ \text{and} \ \ \widehat{\pi}_\gamma = \widehat{\pi}_{\gamma\prm}.
	\] Furthermore, for any $\alpha \in \widehat{\Gamma}$, $m \in \tIalpha \nu k \alpha \gamma$, $n \in \tJalpha \nu k \alpha \gamma$, we have \[
	\falpha \nu k m \alpha \gamma = \falpha \nu k m \alpha {\gamma\prm} \ \ \text{and} \ \ \galpha \nu k n \alpha \gamma = \galpha \nu k n \alpha {\gamma\prm}.
	\]
\end{prop}

\begin{proof}
	Let $f \in M_k^!(\Gamma, \nu)$ be arbitrary, and write $f^\lambda(z) = \sum\limits_n a^\lambda(n) q^n$. Suppose $\cusp {\gamma \infty} = \cusp {\gamma\prm \infty}$; then we may write $\gamma\prm = \pm \errormatrix \gamma \smatrix 1r01$ with $\errormatrix \in \Gamma$ and $r \in \Q$. By Lemma \ref{Lemmamatricesforcusp}, we see $a^{\gamma\prm}(n) = \nu(\errormatrix) \e^{2 \pi \i r n} a^{\gamma}(n)$, and $\Ialpha \nu k \gamma = \Ialpha \nu k {\gamma\prm}$, so by Proposition \ref{FGcusps} we have $\nu(\errormatrix) \e^{2 \pi \i r m} \Fgamma \nu k m {\gamma\prm} = \Fgamma \nu k m {\gamma}$.
	We compute
\begin{align*}
	\widetilde{\pi}_\gamma(f) - \widetilde{\pi}_{\gamma\prm}(f) &= \parent{f - \sum_{m \in \Ialpha {\nu} {k} {\gamma}} a^{\gamma}(m) \Fgamma \nu k m {\gamma}} - \parent{f - \sum_{m \in \Ialpha {\nu} {k} {\gamma\prm}} a^{\gamma\prm}(m) \Fgamma \nu k m {\gamma\prm}} \\
	&= \sum_{m \in \Ialpha {\nu} {k} {\gamma\prm}} a^{\gamma\prm}(m) \Fgamma \nu k m {\gamma\prm} - \sum_{m \in \Ialpha {\nu} {k} {\gamma}} a^{\gamma}(m) \Fgamma \nu k m {\gamma} \\
	&= \sum_{m \in \Ialpha {\nu} {k} {\gamma}} \parent{\nu(\errormatrix) \e^{2 \pi \i r m} a^{\gamma}(m) \Fgamma \nu k m {\gamma\prm} - a^{\gamma}(m) \Fgamma \nu k m {\gamma}} \\
	&= \sum_{m \in \Ialpha {\nu} {k} {\gamma}} \parent{a^{\gamma}(m) \Fgamma \nu k m {\gamma} - a^{\gamma}(m) \Fgamma \nu k m {\gamma}} \\
	&= 0.
\end{align*}
The proof for $\widehat{\pi}_\gamma$ and $\widehat{\pi}_{\gamma\prm}$ is totally analogous. This proves the first claim.

Let $h = \falpha \nu k m \alpha \gamma - \falpha \nu k m \alpha {\gamma\prm}$, and write $h^\lambda (z) = \sum\limits_n a^\lambda(n) q^n$. By construction, $h$ has no poles except possibly for poles of order $N \in \Jalpha {\conj \nu} {2-k} \gamma$ at the cusp $\cusp \gamma\infty$. But for $N \in \Jalpha {\conj\nu} {2-k} \gamma$, we see that
\[
	0 = \set{h, \Ggamma {\conj \nu} {2-k} N \gamma}_\Gamma = \frac{\cuspwidth \Gamma {\gamma \infty}}{\widehat w} a^\gamma(N).
\]
Thus $h$ has no poles at any cusp, and so $h$ is a holomorphic form. Then
\[
	0 = \widetilde{\pi}_\gamma(h) = \widetilde{\pi}_\gamma(\falpha \nu k m \alpha \gamma) - \widetilde{\pi}_{\gamma\prm}(\falpha \nu k m \alpha {\gamma\prm}) = \falpha \nu k m \alpha \gamma - \falpha \nu k m \alpha {\gamma\prm},
\]
which proves the equality. A similar argument shows that $\galpha \nu k n \alpha \gamma = \galpha \nu k n \alpha {\gamma\prm}$.
\end{proof}

Observe that for $\lambda \in \widehat \Gamma$ we have
\begin{align}\label{Eqn_fcusp1}
	\falpha \nu k m \alpha \gamma |_k \lambda = \falpha {\nu^\lambda} k m {\lambda\inv \alpha} {\lambda \inv \gamma} \ \ &\text{and} \ \
	\galpha \nu k m \alpha \gamma |_k \lambda = \galpha {\nu^\lambda} k m {\lambda\inv \alpha} {\lambda\inv \gamma}.
\end{align}
In particular, we see that
\begin{align*}
	\finf \nu k m |_k \gamma = \falpha {\nu^\gamma} k m {\gamma\inv} {\gamma\inv} \ \ &\text{and} \ \
	\ginf \nu k m |_k \gamma = \galpha {\nu^\gamma} k m {\gamma\inv} {\gamma\inv},
\end{align*}
so Theorem \ref{ThmDualityPolesAnywhere} below fulfills the promise made at the beginning of this subsection.

\begin{theorem}\label{ThmDualityPolesAnywhere}
	Let $\Gamma $ be a commensurable subgroup, and let $\widehat{\Gamma}$ be a maximal commensurable subgroup containing $\Gamma$. Let $k\in \frac 1 2 \Z,$ and let $\nu$ be any consistent weight $k$ multiplier. Finally, let $\alpha,$ $\beta$, and $\gamma$ be elements of $\widehat{\Gamma}$. The coefficients of the forms $\parent{\falpha {\nu^\alpha} {k} m {\alpha\inv \beta} {\alpha\inv \gamma} }_m$ and $\parent{\galpha {\conj \nu^\beta} {2-k} n {\beta\inv \alpha} {\beta\inv \gamma}}_n$ satisfy
	\[
		\cuspwidth \Gamma {\alpha \infty} \aalpha {\nu^\alpha} k m n {\alpha\inv \beta} {\alpha\inv \gamma}
		= - \cuspwidth \Gamma {\beta \infty} \balpha {{\conj\nu}^\beta} k n m {\beta\inv \alpha} {\beta\inv \gamma}.
	\]
\end{theorem}

\begin{proof}
	We apply the modular pairing to $\falpha \nu k m \beta \gamma$ and $\galpha {\conj \nu} {2-k} n \alpha \gamma$. By Lemma \ref{Lemma: fmu and gmu exist}, these are both weakly holomorphic functions, and so by Theorem \ref{ThmPairing} we see
	\[
	\set{\falpha \nu k m \beta \gamma, \galpha {\conj \nu} {2-k} n \alpha \gamma}_\Gamma = 0.
\]
We have that $\falpha \nu k m \beta \gamma \galpha {\conj \nu} {2-k} n \alpha \gamma$ vanishes at each cusp except possibly at $\cusp {\alpha \infty}$, $\cusp {\beta \infty}$, and $\cusp {\gamma \infty}$. The contribution at the cusp $\cusp \gamma \infty$ is
\begin{equation}\label{Contribution from gamma infinity}
	\frac{\cuspwidth \Gamma {\gamma \infty}}{\widehat w} \sum\limits_{\ell \in \Zrho \nu {\gamma \infty}} \aalpha {\nu^\gamma} k m \ell {\gamma\inv \beta} I \balpha {{\conj\nu}^\gamma} {2-k} n {-\ell} {\gamma\inv \alpha} I.
\end{equation}
Now if $\ell < 0$ and $\ell \not\in -\Jalpha {\conj\nu} k \gamma$, then $\aalpha {\nu^\gamma} k m \ell {\gamma\inv \beta} I = 0$; otherwise, we have $\balpha {{\conj\nu}^\gamma} {2-k} n {-\ell} {\gamma\inv \alpha} I = 0$. But this covers every $\ell$, so the contribution to the pairing from this sum is zero. The sums
\begin{equation}\label{Contribution from alpha infinity}
	\frac{\cuspwidth \Gamma {\alpha \infty}}{\widehat w} \sum\limits_{\ell \in \Zrho \nu {\alpha \infty}} \aalpha {\nu^\alpha} k m \ell {\alpha\inv \beta} {\alpha\inv \gamma} \balpha {{\conj\nu}^\alpha} {2-k} n {-\ell} {I} {\alpha\inv \gamma}
\end{equation}
and
\begin{equation}\label{Contribution from beta infinity}
\frac{\cuspwidth \Gamma {\beta \infty}}{\widehat w} \sum\limits_{\ell \in \Zrho \nu {\beta \infty}} \aalpha {\nu^\beta} k m \ell {I} {\beta\inv \gamma} \balpha {{\conj\nu}^\beta} {2-k} n {-\ell} {\beta\inv \alpha} {\beta\inv \gamma}
\end{equation}
likewise vanish. The cusps $\alpha \infty$, $\beta \infty$, and $\gamma \infty$ are not necessarily distinct; if some or all of these cusps are equal, then some or all of \eqref{Contribution from gamma infinity}, \eqref{Contribution from alpha infinity}, \eqref{Contribution from beta infinity} may be identified with one another. However, as each of these expressions vanish, this is immaterial.
Then the only contributions to the pairing are $\frac{\cuspwidth \Gamma {\alpha \infty}}{\widehat w} \aalpha {\nu^\alpha} k m n {\alpha\inv \beta} {\alpha\inv \gamma}$ and $\frac{\cuspwidth \Gamma {\beta \infty}}{\widehat w} \balpha {{\conj\nu}^\beta} {2-k} n m {\beta\inv \alpha} {\beta\inv \gamma}$, and so
\[
	\set{\falpha \nu k m \beta \gamma, \galpha {\conj \nu} {2-k} n \alpha \gamma}_\Gamma = \frac{\cuspwidth \Gamma {\beta \infty}}{\widehat w} \balpha {{\conj\nu}^\beta} k n m {\beta\inv \alpha} {\beta\inv \gamma} + \frac{\cuspwidth \Gamma {\alpha \infty}}{\widehat w} \aalpha {\nu^\alpha} {2-k} m n {\alpha\inv \beta} {\alpha\inv \gamma}.
\]

This concludes the proof.
\end{proof}

In particular, $\parent{\cuspwidth \Gamma {\alpha \infty} \falpha {\nu^\alpha} {k} m {\alpha\inv \beta} {\alpha\inv \gamma}, \cuspwidth \Gamma {\beta \infty} \galpha {\conj \nu^\beta} {2-k} n {\beta\inv \alpha} {\beta\inv \gamma}}_{m, n}$ constitutes a modular grid.

\begin{remark}Theorem \ref{ThmDualityPolesAnywhere} may be extended in the same manner that Theorem \ref{ThmVanishingDuality} extends Theorem \ref{IntroThmDualitySimple}/Theorem~\ref{ThmDualityDetail}. We leave the precise statement and the accompanying burden of notation to the interested reader. This coefficient duality is more easily proven using the techniques of Sections \ref{SecGeneratingFunctions} and \ref{SecArithmetic} to modify the grid given by Theorem \ref{ThmDualityPolesAnywhere}, rather than by using the modular pairing.
\end{remark}

\section{Generating functions}\label{SecGeneratingFunctions}

Let $(f_m, g_n)_{m, n}$ be a modular grid indexed by $I$ and $J$, and let $H(z, \tau)$ be a meromorphic function on $\H \times \bbH$. Write $p^n = \e^{2\pi \i n \tau}$.  We say that $H$ is a generating function for $(f_m, g_n)_{m, n}$ if
\[
H(z, \tau) = \sum_{m \in I} f_m(z) p^m \ \text{and} \ H(z, \tau) = -\sum_{n \in J} g_n(\tau) q^n
\]
wherever each sum converges.

For instance, suppose that
\[
H(z,\tau)=F(z) G(\tau),
\]
for some $F(z)=\sum_{n\in \Znu{\nu}} a_nq^n \in M^!_k(\Gamma,\nu)$ and $G(\tau)=\sum_{m\in \Znu{\conj \nu}}b_mp^m\in M^!_{2-k}(\Gamma,\overline \nu)$. Then it can be easily checked that $H$ is a generating function for $f_m=b_m F$ and $g_n=-a_n G$, which form a modular grid; every element of the grid is a multiple of either $F(z)$ or $G(\tau)$.

A similar computation shows that linear combinations of these types of products of the form \[ H(z,\tau)=\sum_{j} c_j F_j(z) G_j(\tau),\] where each $c_j\in \C$, each $F_j \in M^!_k(\Gamma,\nu)$ and each $G_j\in M^!_{2-k}(\Gamma,\overline \nu)$, give modular grids whose rows or columns span finite-dimensional spaces.  Note that the coefficient duality here makes use only of the product of the Fourier expansions of the $F_j$ and $G_j$, and does not use their modularity properties at all. In this case it is not hard to see that each element $f_m$ or $g_n$ of the grid is a linear combination of the $F_j$ or of the $G_j$ respectively. Consequently, the elements $f_m$ (respectively $g_n$) of any grid of this form must have an absolute bound on the order of the pole at any cusp, given by the maximum order of the poles of the $F_j$ (respectively $G_j$) at that cusp.

The modular grids discussed in previous sections of this paper have no bound on the order of their poles.  To construct their generating functions and prove Theorem~\ref{ThmGenFn}, we need rational functions of modular forms with non-trivial denominators.  We restate the theorem here in our current notation for ease of reference.
\begin{theorem}\label{GenFunctions}
Let $\Gamma $ be a commensurable subgroup, let $k\in \frac 1 2 \Z,$ and let $\nu$ be any consistent weight $k$ multiplier. Let $\parent{\finf \nu k m}_m$ and $\parent{\ginf {\conj \nu} {2-k} n}_n$ be row-reduced canonical bases for the spaces $\spcfinf k \parent{\Gamma, \nu}$ and $\spcginf {2-k} \parent{\Gamma, \conj \nu}$, as in Section \ref{SecDualityWHMF}. There is an explicit generating function $\calH_{k}^{(\nu)}(z, \tau)$ encoding this modular grid; that is, on appropriate regions of $\H \times \H$, the function $\calH_{k}^{(\nu)}(z, \tau)$ satisfies the two equalities
\[
	\calH_{k}^{(\nu)}(z, \tau) = \sum_{m \in \tI \nu k} \finf \nu k m(z) p^m \ \text{and} \ \calH_{k}^{(\nu)}(z, \tau) = -\sum_{n \in \tJ {\overline \nu} {2 - k}} \ginf {\conj \nu} {2 - k} n(\tau) q^n.
\]
The function $\calH_{k}^{(\nu)}(z, \tau)$ is meromorphic on $\H \times \H$ with simple poles when $z = \gamma \tau$ for some $\gamma \in \Gamma$. It is modular in $z$ of weight $k$ with multiplier $\nu$, and modular in $\tau$ of weight $2 - k$ with multiplier $\overline{\nu}$.
\end{theorem}

\begin{proof}
	The forms $\parent{\finf \nu k m}_{m}$ and $\parent{\ginf {\overline\nu} {2-k} n}_{n}$, which constitute bases for the spaces $\spcfinf k \parent{\Gamma, \nu}$ and $\spcginf {2-k} \parent{\Gamma, \overline\nu}$ respectively, may be collected together in a generating series as follows. Let $\varsigma$ and $w$ be the constants defined in Section \ref{Sec:Expansions}.
We define $\finfgen \nu k (z, \tau)$ to be the meromorphic function on $\H \times \H$ given by
\begin{align}\label{FGenDef}
	\finfgen \nu k (z,\tau):&= \frac{p^{\frac{1 - \varsigma} w}q^{\frac {\varsigma} w}}{q^{\frac1w}-p^{\frac 1 w} } + \sum\limits_{m \in \I {\nu} {k}}p^{-m}q^{m} - \sum\limits_{n \in \J {\conj \nu} {2-k}}p^n q^{-n} +\sum\limits_{m \in \tI \nu k} \sum\limits_{n \in \tJ {\overline \nu} {2-k}} \ainf {\nu} {k} mn p^mq^n.
	\end{align}
Note that the two middle sums are sums over finite sets whose sizes are the dimensions of spaces of \emph{holomorphic} modular forms.

If $\im(\tau)>\im (z)$, then we may expand $\finfgen \nu k (z,\tau)$ as a series in powers of $p$. The first three terms become $\sum_{m\in \tI \nu k} q^{-m} p^m$,
and we find that
\begin{equation}\label{FGenF}
	\finfgen \nu k (z, \tau) = \sum_{m \in \tI \nu k} \finf \nu k m (z)p^m.
\end{equation}
\noindent If $\im(z) > \im (\tau)$, then we may expand $\finfgen \nu k (z,\tau)$ as a series in powers of $q$.  The first three terms become $-\sum_{n \in \tJ {\overline \nu} {2-k}} p^{-n} q^n$, and we use Theorem~\ref{IntroThmDualitySimple} to find that
\[
	\finfgen \nu k (z, \tau) = - \sum_{n \in \tJ {\overline \nu} {2-k}} \ginf {\overline\nu} {2-k} n (\tau)q^n.
\]
\noindent
Thus, $\finfgen \nu k (z, \tau)$
is a two-variable meromorphic modular form which transforms in $z$ with weight $k$ and multiplier $\nu$, and in $\tau$ with weight $2-k$ and multiplier $\overline \nu.$

Using the modularity of the coefficients, it is clear that $\finfgen \nu k (z, \tau)$ is holomorphic in $z$ at each cusp $\cusp \rho\neq \infty$, and vanishes as $\tau$ approaches any cusp $\cusp \rho\neq \infty.$
\begin{remark}
If $f(\tau)=\sum_{n\in \Znu \nu} a(n) q^n \in M^!_{k}(\Gamma, \nu)$, with $k, \Gamma,$ and $\nu$ as above, then the circle method (see for instance section IV of \cite{Rademacher}) or the method of Maass-Poincar\'e series (see for instance Section 1.3 of~\cite{Bruinier2002}) can be used to bound the size of the coefficients $a(n).$ In particular, suppose that $M_f$ is the maximum order of any pole of $f$ at any cusp.  If $n, M_f>0,$ then there are constants $c_f, c_1, c_2>0$ such that
\[
|a(n)| <  c_f\, c_1\,  \left(\tfrac{n}{M_f}\right)^{\frac{k-1}{2}} \exp\left(c_2 \sqrt{M_f n}\right),
\]
where $c_f$ is bounded by the sum of the absolute values of the coefficients of the principal part of $f$, and $c_1$ and $c_2$ depend only on the group and multiplier. If $M_f\leq 0$ (i.e. $f$ is a holomorphic modular form), then $|a(n)|$ exhibits a stronger\emph{ polynomial} bound in terms of $n$, whereas if $n\leq 0$, then we clearly have $|a(n)|\leq c_f$. Convergence of Maass--Poincar\'e series is delicate for $1/2\leq k\leq 3/2$; however, similar asymptotic bounds may be established by considering powers of $f$ or multiplication of $f$ by another modular form to change the weight.
Together these bounds suffice to show that the double series in~\eqref{FGenDef} converges absolutely for all $q,p$ with $|q|,|p|<1$.  Further details of the circle method or Maass-Poincar\'e series are beyond the scope of this paper. However, the process we demonstrate in the rest of this section may be followed by considering this summation as a formal sum.
\end{remark}

From the definition it is clear that $\finfgen \nu k (z, \tau)$ has a pole when $z=\tau$. By modularity, $\finfgen \nu k (z, \tau)$ must have a pole at all points $(z,\tau)$ where $z=\gamma \tau$ for some $\gamma\in \Gamma$. We will see later that these poles are simple and that $\finfgen \nu k (z, \tau)$ has no other poles on $\H\times \H$.

We wish to write $\finfgen \nu k (z, \tau)$ as an explicit rational function of weakly holomorphic modular forms. We begin by selecting any non-constant function
\begin{equation}\label{H_def}
H_\Gamma(z) = \sum_{n\in\Znu{\nu_0}}
C_n q^{n}.\end{equation}
which is modular for $\Gamma$ with weight $0$ and \emph{trivial} multiplier $\mathbf{1}$,  so $\Znu{\mathbf{1}}=\frac{1}{w}\Z$. For instance, we may choose $H_\Gamma = \finf {\nu_0} 0 m \in \spcfinf 0 (\Gamma)$ for an appropriate index $m$.  Choosing $m$ to be minimal simplifies the computations.

\begin{remark}
	Our method depends on our choice of $H_\Gamma$, and so the resulting expression for $\finfgen \nu k (z, \tau)$ is far from unique.  However, if $X(\Gamma)$ has genus zero, we may make a canonical choice for $H_\Gamma$ to be a Hauptmodul.
\end{remark}

Now consider the function
\begin{align}\label{G2Def}
	\calK(z,\tau) = (H_\Gamma(z)-H_\Gamma(\tau))\cdot \finfgen \nu k (z, \tau).
\end{align}
Let $\widetilde f_m(z)$ and $\widetilde g_n(\tau)$ be the coefficient functions of $\calK(z,\tau)$ when expanded in powers of $p$ or $q$ respectively, so that on appropriate regions of $\H\times \H$, we have the equalities
\[	
	\calK(z,\tau) = \sum_{m \in \tI \nu k} \widetilde f_m (z)p^m \ \text{and} \ \calK(z,\tau) = \sum_{n \in \tJ {\overline \nu} {2-k}} \widetilde g_n(\tau)q^n.
\]
\noindent Expanding \eqref{G2Def}, we find that
\begin{align}
\label{ft}	\widetilde f_m &= \finf \nu k m\cdot H_\Gamma-\sum_{i\in \tI \nu k} C_{m-i}\cdot \finf \nu k {i} ,\\
\label{gt}	\widetilde g_n &= \ginf {\overline\nu} {2-k} n \cdot H_\Gamma -\sum_{j\in \tJ {\overline\nu} {2-k}} C_{n-j}\cdot \ginf {\overline\nu} {2-k} j.
\end{align} Note that these are finite sums.

We make three observations.
\begin{enumerate}
\item[1)]Each $\widetilde f_m$ is a form in $M^!_k(\Gamma,\nu)$, and each $\widetilde g_n$ is a form in $M^!_{2-k}(\Gamma,\overline\nu)$.
\item[2)] If $\cusp\rho$ is any cusp of $\Gamma$ other than $\cusp\infty$, then the order of the pole of $\widetilde f_m$ and $\widetilde g_m$ at $\rho$ is bounded by the order of the pole of $H_\Gamma$ at $\rho.$ In particular, if $H_\Gamma\in \spcfinf 0 \parent{\Gamma},$ then $\widetilde f_m\in \spcfinf k (\Gamma,\nu)$ and $\widetilde g_n \in \spcginf {2-k}(\Gamma,\overline \nu).$

\item[3)] There is a global bound on the orders of the poles of the forms appearing as the coefficients of $\calK(z,\tau)$ expanded in either variable.  If $B_\infty$ is the order of the pole of $H_\Gamma$ at $\infty$, then the order of the pole of $\tilde f_m$ at infinity is bounded by $B_\infty+\max (\J{\overline\nu} {2-k})$, and
the order of the pole of $\tilde g_n$ at infinity is bounded by $B_\infty+\max (\I{\nu} {k})$. This follows either by carefully expanding \eqref{ft} and \eqref{gt}, or by substituting \eqref{FGenDef} into \eqref{G2Def}. In the second case, the sums in \eqref{FGenDef} multiplied by $H_\Gamma(z)-H_\Gamma(\tau)$ clearly exhibit these bounds on the negative exponents of $p$ and $q$. In order to find the bounds for the initial term of \eqref{FGenDef} multiplied by $H_\Gamma(z)-H_\Gamma(\tau),$ we expand
\begin{align*}
	\frac{H_\Gamma(z)-H_\Gamma(\tau)}{q^{\frac 1w}-p^{\frac 1 w} }
	&=\sum_{\substack{n\in \Znu{\nu_0}\\ n\geq -B_\infty}} C_{n}\frac{q^{n}-p^{n}}{q^{\frac 1w}-p^{\frac 1 w} }
	\\	&= \sum_{\substack{n\in \Znu{\nu_0}\\-B_\infty\leq n<0}} C_{n}
	\sum_{\substack{m\in \Znu{\nu_0}\\n\leq m<0}} -q^{m}p^{n-m-\frac {1}w}
	+
	\sum_{\substack{n\in \Znu{\nu_0}\\n>0}} C_{n}
	\sum_{\substack{m\in \Znu{\nu_0}\\ 0\leq m< n}} q^{m}p^{n-m-\frac {1}w}.
\end{align*}
Here, the $C_n$ are the coefficients of $H_\Gamma$ as defined in \eqref{H_def}. The exponents in either variable of this final expression are bounded below by $-B_\infty.$

\end{enumerate}
Given these three observations, we see that $\calK (z,\tau)$ lives in the finite-dimensional subspace of $M^!_k(\Gamma,\nu)\otimes M^!_{2-k}(\Gamma,\overline \nu)$ generated by forms in each variable whose principal parts at each cusp have order satisfying the bounds described above. Since this space is finite-dimensional, it is a straightforward linear algebra problem to calculate $\calK(z,\tau)$ in terms of basis elements of $M^!_k(\Gamma,\nu)$ and $M^!_{2-k}(\Gamma,\overline \nu)$ respectively.
In particular, we note that if $H_\Gamma\in \spcfinf 0 \parent{\Gamma}$ (with trivial multiplier), then
\[
	\calK (z,\tau)\in \spcfinf k \parent{\Gamma, \nu}\otimes \spcginf {2-k}\parent{\Gamma, \overline \nu},
\]
and so we may write $\calK (z,\tau)$ as a finite sum of the form $\sum\limits_{j} c_j \finf \nu k {m_j} (z) \ginf {\conj \nu} {2-k} {n_j} (\tau)$ for some $c_j$, $m_j$, and $n_j$.

We now have
\[
	\finfgen \nu k(z,\tau)=\frac{\calK(z,\tau)}{H_\Gamma(z)-H_\Gamma(\tau)}.
\]
Since $\calK(z,\tau)\in M^!_k(\Gamma,\nu)\otimes M^!_{2-k}(\Gamma,\overline \nu),$ the only poles of $\finfgen \nu k(z,\tau)$ in $\H\times \H$ occur when $H_\Gamma(z)-H_\Gamma(\tau)=0.$ Since $H_\Gamma$ was chosen arbitrarily in $M^!_0(\Gamma)$, the poles of $\finfgen \nu k(z,\tau)$ occur only when $z=\gamma\tau$ for some $\gamma\in \Gamma$, and these poles are simple.
\end{proof}

Analogous generating functions for the forms appearing in Theorem~\ref{ThmDualityPolesAnywhere} could be constructed following a similar process, but would require special care to handle the poles at different cusps. A much simpler approach is demonstrated in Proposition \ref{GenAct}, by considering matrix actions on the generating functions given in Theorem \ref{GenFunctions}.

\section{Arithmetic operations on modular grids}\label{SecArithmetic}

Modular grids can be combined and altered by linear operators. In this section, we give three examples of this phenomenon.

\subsection{Linear combinations of grids}\label{Subsection: linear combinations of grids}

At the beginning of Section \ref{SecGeneratingFunctions}, we discussed modular grids of the shape $ H(z,\tau) = \sum_{j} c_j F_j(z) G_j(\tau)$, where each $c_j \in \C$, each $F_j \in M^!_k(\Gamma,\nu)$ and each $G_j \in M^!_{2-k}(\Gamma,\overline \nu)$. More generally, linear combinations of modular grids are modular grids.

\begin{proposition}\label{Proposition: Linear combinations of modular grids}
Assume the notation above. Let $(f_{m}^{(1)},g_{n}^{(1)})_{m, n}, \ldots, (f_{m}^{(N)},g_{n}^{(N)})_{m, n}$ be a finite list of modular grids with identical indexing sets, and suppose $f_{m}^{(j)}\in M^!_k(\Gamma,\nu)$ and $g_{n}^{(j)}\in M^!_{2-k}(\Gamma,\overline \nu)$ for $j \in \set{1,2, \ldots, N}$. For any fixed $c_1,c_2, \ldots, c_N\in \C$, if we write
\[
	f_{m} = \sum\limits_{j = 1}^N c_j f_{m}^{(j)} \ \ \ \text{and} \ \ \ g_{n} = \sum\limits_{j = 1}^N c_j \, g_{n}^{(j)},
\]
then $(f_{m},g_{n})_{m, n}$ is a modular grid.
\end{proposition}

\begin{proof}
	Let $(f_{m}^{(1)},g_{n}^{(1)})_{m, n}, \ldots, (f_{m}^{(N)},g_{n}^{(N)})_{m, n}$ be as above; for $j \in \set{1,2, \ldots, N}$, write
	\[
		f_m^{(j)}(z) = R_m^{(j)}(q) + \sum\limits_{n \in J} a^{(j)}(m, n) q^n \ \ \ \text{and} \ \ \ g_n^{(j)}(\tau) = S_n^{(j)}(p) + \sum\limits_{m \in I} b^{(j)}(n, m) p^m,
	\]
	where $R_{m}^{(j)}$ and $S_{n}^{(j)}$ are Laurent polynomials in fractional powers of $q$ and $p$ respectively. Set
	\begin{align*}
		R_m(q) &= \sum\limits_{j = 1}^N c_j R_m^{(j)}(q), \\
		S_n(p) &= \sum\limits_{j = 1}^N c_j S_n^{(j)}(p), \\
		a(m, n) &= \sum\limits_{j = 1}^N c_j a^{(j)}(m, n), \ \text{and} \\
		b(n, m) &= \sum\limits_{j = 1}^N c_j b^{(j)}(n, m).
	\end{align*}
	Then by construction,
	\begin{align*}
	f_{m}(z) &= \sum\limits_{j = 1}^N c_j f_{m}^{(j)}(z) = R_m(q) + \sum\limits_{n \in J} a(m, n) q^n, \\
	g_n(\tau) &= \sum\limits_{j = 1}^N c_j \, g_{n}^{(j)}(\tau) = S_n(p) + \sum\limits_{m \in I} b(n, m) p^m, \ \text{and} \\
	a(m, n) &= \sum\limits_{j = 1}^N c_j a^{(j)}(m, n) = - \sum\limits_{j = 1}^N c_j b^{(j)}(n, m) = - b(n, m);
	\end{align*}
	our claim follows.
\end{proof}

\begin{example} (See~\cite{Ahlgren}.)  Let $\ell$ be any prime such that $X_0(\ell)$ has genus $0$, and let $J_\ell$ be the Hauptmodul for $X_0(\ell)$ with a pole at $\infty$. Let
\[
	\calH_1(z,\tau)=\frac{\frac{1}{2\pi \i}J'_\ell(\tau)}{J_\ell(\tau)-J_\ell(z)}.
\]
The coefficients of $\calH_1$, expanded in $p$ and $q$, form a modular grid $(f_{m}^{(1)},g_{n}^{(1)})_{m\geq 0,n\geq 1}$, with $f_0^{(1)}=1$, $f_{m}^{(1)}=q^{-m}+O(q)$, and $g_n^{(1)}=p^{-n}+O(p).$ The $f_m^{(1)}$ forms create a basis for $M^{(\infty)}_{0}(\Gamma_0(\ell))$ (with trivial multiplier), while the $g_n^{(1)}$ forms create a basis for $\widehat M^{(\infty)}_{2}(\Gamma_0(\ell)).$ The function
\[
	\calH_2(z,\tau)= \widehat E_{2,\ell}(\tau)=\frac{1}{1-\ell}\left(E_2(\tau)-\ell E_2(\ell\tau)\right)
	\]
is also the generating function of a modular grid for the same groups with the same indexing sets. In this case, the weight $0$ functions of the grid are constants.
Then by Proposition \ref{Proposition: Linear combinations of modular grids}, the coefficients of
\[
	\calH_3(z,\tau) = \calH_1(z,\tau)-\calH_2(z,\tau)=\frac{\frac{1}{2\pi i}J'_\ell(\tau)}{J_\ell(\tau)-J_\ell(z)}-\widehat E_{2,\ell}(\tau)
\]
form a new modular grid $(f_{m},g_{n})_{m\geq 1,n\geq 0}.$ Notice that $g_{n}=g_{n}^{(1)}$ for $n\geq 1$, and $g_{0}=\widehat E_{2,\ell}(\tau),$ whereas $f_{m}=f_{m}^{(1)} - c_\ell(m),$ where $c_\ell(m)$ is the coefficient of $p^m$ of $\widehat E_{2,\ell}(\tau).$ By considering the modular pairing $\{f_{m}, \widehat E_{2,\ell}\}_\Gamma,$ we find that $f_0(z)=0$, and for $m\geq 1$, the constant term of $f_{m}$ at the cusp $0$ now vanishes, rather than the constant term at $\infty$.
\end{example}

\subsection{Modular transformations of grids}
We may also transform a modular grid arising from a rational generating function by the action of matrices on one or both variables.  For example, consider the modular grid $\parent{\finf \nu k m, \ginf {\conj \nu}{2-k} n}_{m, n}$ with associated bivariate generating function $\finfgen \nu k (z, \tau)$. Rather than acting on the $\finf \nu k m$ and the $\ginf {\conj \nu} {2 - k} n$ by a modular transformation, as described before Question~\ref{q1}, we apply the usual weight $k$ slash operator~$|_{k,w}$ in the $w$ variable (either $z$ or $\tau$), as defined in \eqref{SlashDef}, to their joint generating function. This process yields generating functions for the modular grids described by Theorem \ref{ThmDualityPolesAnywhere}, and thus vindicates its generality: Theorem \ref{ThmDualityPolesAnywhere} describes the behavior of the modular grid $\parent{\finf \nu k m (z), \ginf {\conj \nu} {2-k} n(\tau)}_{m, n}$ under modular transformations.

\begin{proposition}\label{GenAct}
Assume the notation above. Suppose $\finfgen \nu k (z, \tau)$ encodes the grid $\parent{f_{k,m}^{(\nu)},g_{2-k,n}^{(\overline\nu)}}_{m, n}$, as defined in Section \ref{SecGeneratingFunctions}, and suppose $\alpha, \beta,\gamma \in\widehat\Gamma$.
Then the function \[ \finfgen {\nu^\gamma} k (z, \tau)|_{k,z}(\gamma^{-1}\alpha)\,|_{2-k,\tau}(\gamma^{-1}\beta) \]
encodes the modular grid
\[
\left( \frac{\cuspwidth \gamma {\gamma\infty}}{\cuspwidth \beta{\beta\infty}}\falpha {\nu^\alpha} k m {\alpha\inv\beta} {\alpha\inv\gamma} (z), \
\frac{\cuspwidth \gamma {\gamma\infty}}{\cuspwidth \alpha {\alpha\infty}}\galpha {\conj\nu^{\beta}} {2-k} n {\beta\inv \alpha} {\beta\inv\gamma}(\tau)\right)_{m, n}.\]
\end{proposition}

\begin{proof} Define the function $\mathcal H'(z,\tau) = \finfgen {\nu^\gamma} k |_{k,z}\gamma\inv\alpha (z,\tau).$ This function is modular in the variable $z$ with respect to the multiplier $\nu^\alpha$, and  in the variable $\tau$ with respect to the multiplier  $\bar \nu^\gamma.$ Thus $\mathcal H'$ has a Fourier expansion in fractional powers of $q$ and $p$, with admissible exponents in $\Znu{\nu^\alpha}$ and $\Znu{\bar \nu^\gamma}$ respectively. The Fourier coefficients in one variable may be extracted by standard Fourier integrals in that variable. Specifically, 
for $m \in \Znu {\bar\nu^\gamma} $ and $n \in \Znu {\nu^\alpha}$, write
\begin{align*}
f'_m(z)=\lim_{\tau_0\to \i \infty}\int_{\tau_0}^{\tau_0+w_{\gamma\infty}} \frac{p^{-m}}{w_{\gamma\infty}}\mathcal H'(z,\tau)  \,d\tau, \\
g'_n(\tau)=-\lim_{z_0\to \i \infty}\int_{z_0}^{z_0+w_{\alpha\infty}} \frac{q^{-n}}{w_{\alpha\infty}}\mathcal H'(z,\tau)  \,dz.
\end{align*}
These functions are modular for their respective multipliers since modular transformations commute with the integral. The limit is required in the definition since the function  $\mathcal H'(z,\tau)$ has poles. If $z_0$ and $\tau_0$ are taken as variables while $z$ and $\tau$ are held constant, the integrals above may be discontinuous as the paths of integration cross any singularities.

The forms $\parent{f'_m, g'_n}_{m, n}$ constitute a new modular grid encoded by $\mathcal H'(z,\tau)$.
 Duality follows from the fact that in general the order of the Fourier integrals
\[\lim_{\tau_0\to \i \infty}\int_{\tau_0}^{\tau_0+w_{\gamma\infty}} \frac{p^{-m}}{w_{\gamma\infty}}\bullet  \,d\tau \ \ \text{and} \ \
\lim_{z_0\to \i \infty}\int_{z_0}^{z_0+w_{\alpha\infty}} \frac{q^{-n}}{w_{\alpha\infty}}\bullet  \,dz\]
are interchangeable. This statement is complicated somewhat due to the presences of poles in either variable whose locations are determined by the other variable. However, this complication can be resolved without too much difficulty.

To see this, recall that $\finfgen {\nu^\gamma} k$ has poles only at cusps and when $z=\mu \tau$ for some $\mu\in \Gamma^{\gamma}.$ Thus
$\mathcal H'$ has poles only at cusps or when $\alpha\inv \gamma z=\mu \tau$ for some $\mu\in \Gamma^{\gamma}.$ If the cusps
$\alpha\infty $ and $\gamma\infty$ are not equivalent, then the function is holomorphic for $z$ and $\tau$ in $\H$ with sufficiently large imaginary part, and so the integrals are interchangeable.

If the cusps $\alpha\infty $ and $\gamma \infty$ are equivalent, then by Lemma \ref{Lemmamatricesforcusp}, $\alpha\inv \gamma=\beta T$ for some $\beta \in \Gamma^\gamma$ and some upper triangular matrix $T$. Following \eqref{FGenDef}, let $\varsigma'=\varsigma(\nu^\gamma)$ and $w'=w(\nu^\gamma)=w_{\gamma\infty}.$
Then we see that
\[
\finfgen {\nu^\gamma}k (z,\tau) -\frac{p^{\frac{1 - \varsigma'} {w'}}q^{\frac {\varsigma'} {w'}}}{q^{\frac1{w'}}-p^{\frac 1 {w'}} }
\]
is holomorphic for $z$ and $\tau$ in $\H$ with sufficiently large imaginary parts. Thus, the same is true for
\[
\mathcal H'(z,\tau) -\nu^\gamma(\beta) \left. \frac{p^{\frac{1 - \varsigma'} {w'}}q^{\frac {\varsigma'} {w'}}}{q^{\frac1{w'}}-p^{\frac 1 {w'}}} \right|_{k,z} T.
\]
Therefore the order of the Fourier integrals for this function are also interchangeable. In particular, the $n$th Fourier coefficient of $f'_m$ is equal to the $m$th Fourier coefficient of $-g'_n$,
except for when $\alpha\inv \gamma$ corresponds to the infinite cusp and we consider coefficients belonging to the rational component $\nu^\gamma(\beta) \left. \frac{p^{\frac{1 - \varsigma'} {w'}}q^{\frac {\varsigma'} {w'}}}{q^{\frac1{w'}}-p^{\frac 1 {w'}}} \right|_{k,z} T $.

Given that $\mathcal H'(z,\tau)$ encodes a modular grid, \eqref{FGenF} and \eqref{Eqn_fcusp1} allow us to determine that  \[f'_m = \falpha {\nu^\alpha} k m {\alpha\inv\gamma} {\alpha\inv\gamma}. \]
Theorem \ref{ThmDualityPolesAnywhere} can then be used to determine that 
$\finfgen {\nu^\gamma} k (z, \tau)|_{k,z}\gamma\inv\alpha $
 encodes the grid
\[
	\left( \falpha {\nu^\alpha} k m {\alpha\inv\gamma} {\alpha\inv\gamma} (z), \ \frac{\cuspwidth {\gamma} {\gamma\infty}}{\cuspwidth \alpha {\alpha\infty}}\galpha {\conj\nu^\gamma} {2-k} n {\gamma\inv\alpha} {I}(\tau)\right)_{m, n}.
\]
Applying this process a second time, this time acting in the variable $\tau$ by the matrix $\gamma\inv\beta$, yields the result.
\end{proof}

We obtain an extra level of symmetry when there are no holomorphic forms in either space to row-reduce against, and we use two matrices $\alpha$ and $\beta$ which both commute with the group and multiplier.

\begin{theorem}\label{ThmMatrixSymm}
Let $k, \nu, \Gamma$, and $\widehat \Gamma$ be given as above, and suppose that $M_k(\Gamma,\nu)$ and $S_{2-k}(\Gamma,\overline\nu)$ are both trivial.
Suppose $\alpha, \beta \in \widehat \Gamma$ satisfy $\alpha \Gamma = \beta \Gamma$, $\Gamma^{\alpha}=\Gamma=\Gamma^{\beta}$, and $\nu^\alpha=\nu=\nu^\beta$. Then
\begin{equation*}
	\finfgen \nu {k} (z, \tau)\big|_{k,z}\alpha=\finfgen \nu {k} (z, \tau)\big|_{2-k,\tau}\beta\inv.
\end{equation*}
\end{theorem}

\begin{proof}
Proposition \ref{GenAct} shows that
\[ \finfgen {\nu} k (z, \tau)|_{k,z}\alpha\,|_{2-k,\tau}\beta \]
encodes the grid
\begin{equation}\label{72Display}
\left( \frac{\cuspwidth I {\infty}}{\cuspwidth \beta{\beta\infty}}\falpha {\nu^\alpha} k m {\alpha\inv\beta} {\alpha\inv} (z), \
\frac{\cuspwidth I {\infty}}{\cuspwidth \alpha {\alpha\infty}}\galpha {\conj\nu^{\beta}} {2-k} n {\beta\inv \alpha} {\beta\inv}(\tau)\right)_{m, n}.
\end{equation}
Let $\gamma=\alpha\inv \beta$. Since $\Gamma^\alpha=\Gamma,$ and $\gamma\in \Gamma,$ we have that $\cuspwidth \beta {\beta\infty}=\cuspwidth \alpha {\alpha\infty}=\cuspwidth I {\infty}$. Thus the grid in \eqref{72Display} is equal to the grid
\begin{equation*}
	\left( \falpha {\nu} k m {\gamma} {\alpha\inv} (z), \
	\galpha {\conj\nu} {2-k} n {\gamma\inv } {\beta\inv}(\tau)\right)_{m, n}.
\end{equation*}
Since $M_k(\Gamma,\nu)$ and $S_{2-k}(\Gamma,\overline\nu)$ are both trivial, and using again that $\gamma\in \Gamma$, we have that $\falpha {\nu} k m {\gamma\inv} {\alpha\inv} (z) = \finf \nu k m$ and $\galpha {\conj \nu} {2-k} n {\gamma}{\beta\inv}(\tau)=\ginf {\conj\nu} {2-k} n$. Therefore
\[ \finfgen \nu k (z, \tau)|_{k,z}\alpha|_{2-k,\tau}\beta = \finfgen \nu k (z, \tau).\]
The result follows.
\end{proof}

Theorem \ref{ThmMatrixSymm} can be adapted to include certain trace operations between groups. Suppose that $\Gamma'$ is a subgroup of $\widehat \Gamma$ satisfying
\[
	\Gamma\subseteq\Gamma'\subseteq\widehat \Gamma,
\]
and that $\nu$ extends to a consistent multiplier $\nu'$ on $\Gamma'$. Then define the trace $\tr \nu {\nu'} :M_k^!(\Gamma,\nu)\to M_k^!(\Gamma',\nu') $
by
\[
f|_k \tr \nu {\nu'} = \sum_{[\gamma]\in \Gamma \backslash \Gamma'} \nu'(\gamma)^{-1} f|_k\gamma.
\]
This is well-defined as long as $f$ is modular for $\Gamma$ with multiplier $\nu$.

With this definition, we have the following theorem.
\begin{theorem}\label{ThmTrace}
Let $\Gamma $ be a commensurable subgroup, and let $\widehat{\Gamma}$ be a maximal commensurable subgroup containing $\Gamma$. Let $k\in \frac 1 2 \Z,$ and let $\nu$ be any consistent weight $k$ multiplier. Suppose that $M_k(\Gamma,\nu)$ and $S_{2-k}(\Gamma,\overline\nu)$ are both trivial. If $\nu$ extends to a multiplier $\nu'$ defined on a group $\Gamma'$ between $\Gamma$ and $\widehat\Gamma,$ then
\[
\finfgen \nu k (z, \tau)|_{k,z} \tr\nu{\nu'}\, = \, \finfgen \nu k (z, \tau)|_{2-k,\tau} \tr{\conj \nu}{\conj\nu'}\, = \, \frac{\cuspwidth{I}{\infty}}{\cuspwidth{I}{\infty}'}\finfgen {\nu'} k (z, \tau),
\]
where $\cuspwidth{I}{\infty}$ and $\cuspwidth{I}{\infty}'$ are the widths of the infinite cusps of $\Gamma$ and $\Gamma'$ respectively.
\end{theorem}

\begin{proof}
Since $M_k(\Gamma,\nu)$ and $S_{2-k}(\Gamma,\overline\nu)$ are both trivial, it follows that $M_k(\Gamma',\nu')$ and $S_{2-k}(\Gamma',\overline\nu')$ are both trivial as well because the latter spaces are subspaces of the former. Therefore we have that
\begin{align*}
\tI \nu k &=\Znu {\conj \nu}_{>0}, & \tJ {\conj\nu} {2-k} &=\Znu {\nu}_{\geq 0},\\
\tI {\nu'} k &=\Znu {\conj \nu'}_{>0}, & \tJ {\conj\nu'} {2-k} &=\Znu {\nu'}_{\geq 0}.
\end{align*}
Suppose $m\in \tI \nu k$. Then $\finf \nu k m=q^{-m}+O(1)$.  If $[\gamma]\in \Gamma\backslash \Gamma'$ does not preserve the cusp $\infty$ of $\Gamma,$ then $\finf \nu k m|_{k,\nu'}\gamma =O(1).$ If
  $[\gamma]\in \Gamma\backslash \Gamma'$ preserves the cusp $\infty$ of $\Gamma$, then we may choose $\gamma$ in the coset to be upper triangular. Therefore
${\nu'}\inv(\gamma) \finf \nu k m |_k \gamma = \zeta q^{-m}+O(1),$ where $\zeta$ is some root of unity and $\zeta=1$ if $m\in \tI {\nu'} k$.

Since $\Gamma\subseteq\Gamma'$, we have that $\Znu {\nu}\subseteq\Znu {\nu'}$ and $\Znu {\conj\nu}\subseteq\Znu {\conj\nu'}$ with
\[
[\Znu {\nu'}:\Znu {\nu}] \, =\, [\Znu {\conj\nu'}:\Znu {\conj\nu}] \, =\, \frac{\cuspwidth{I}{\infty}}{\cuspwidth{I}{\infty}'},
\]
where $[\Znu {\nu'}:\Znu {\nu}]$ is the index of $\Znu {\nu}$ inside $\Znu {\nu'}$ as (possibly shifted) lattices.
The trace $\tr \nu {\nu'}$ will include $[\Znu {\nu'}:\Znu {\nu}]$ matrices which preserve the cusp $\infty$ of $\Gamma.$
Thus, if $m$ is in $\tI {\nu'} k$, then it follows that the form $\finf \nu k m |_k \tr \nu {\nu'}$ has Fourier expansion \[\finf \nu k m |_k \tr \nu {\nu'} = [\Znu {\nu'}:\Znu {\nu}] q^{-m}+O(1)\] and is in $\spcfinf k (\Gamma', \nu')$. Since the space has no holomorphic forms, it must be true that
\[\finf \nu k m |_k \tr \nu {\nu'} = \frac{\cuspwidth{I}{\infty}}{\cuspwidth{I}{\infty}'}\finf {\nu'} k m.
\]

On the other hand, if $m$ is not in $\tI {\nu'} k$, then $\finf \nu k m |_k \tr \nu {\nu'} = O(1)$ and is in $\spcfinf k (\Gamma', \nu').$ Since the space has no holomorphic forms, it follows that
\[
	\finf \nu k m |_k \tr \nu {\nu'} = 0.
\]

This implies that
\[
	\finfgen \nu k (z, \tau)|_{k,z} \tr\nu{\nu'}\, = \, \frac{\cuspwidth{I}{\infty}}{\cuspwidth{I}{\infty}'}\finfgen {\nu'} k (z, \tau).
\]
A similar argument, considering the functions $\ginf {\conj\nu} {2-k} n |_k \tr {\conj\nu} {\conj\nu'}$, shows that
\[
	\finfgen \nu k (z, \tau)|_{2-k,\tau} \tr{\conj \nu}{\conj\nu'}\, = \, \frac{\cuspwidth{I}{\infty}}{\cuspwidth{I}{\infty}'}\finfgen {\nu'} k (z, \tau),
\]
as desired. In this case we may replace each $O(1)$ in the argument above with $o(1)$ since the forms vanish away from $\infty$ and have no constant other than for the $n=0$ form.
\end{proof}

\subsection{Hecke operators}
If $\mathcal F(z,\tau)$ generates a modular grid for a congruence subgroup $\Gamma$, then by combining the ideas of the previous subsections, we can act on $\mathcal F$ by Hecke operators. When there are no cusp forms in the ambient spaces of modular forms, this action demonstrates a striking symmetry. For simplicity, we restrict our attention to $k$ an even integer, $\Gamma\supseteq\Gamma_0(N)$ for some positive integer $N$, and a multiplier $\nu$ which acts trivially on $\Gamma_0(N)$. If $M$ is coprime to $N$, then the action of the Hecke operator $T_M$ on $M^!_k(\Gamma,\nu)$ is an endomorphism given by
\begin{equation}\label{HeckeDef}
	f\big|_k T_M= \sum_{\substack{ad=M\\
	b \bmod{d} }} f\big|_{k} \left(\begin{smallmatrix}a&b\\0&d\end{smallmatrix}\right).
\end{equation}
\begin{remark}
Our normalization of the Hecke operator is $M^{1-k/2}$ times the usual Hecke operator that preserves integrality of coefficients when $k>0$.
\end{remark}
\noindent If $f(z)=\sum_{n}a(n)q^n,$ then
\begin{equation}\label{EqnHecke}
f(z)\big|_k T_M= \sum_{n} \sum_{d\,|\,\gcd(n,M)} d^{k/2}\left(\tfrac{M}{d}\right)^{1-k/2}a\parent{\tfrac{Mn}{d^2}} q^n.
\end{equation}

We may apply this operator to the generating function $\finfgen \nu {2-k} (z, \tau)$ with respect to either variable. In either case this will result in a new modular grid.
As claimed in Theorem \ref{HeckeSym}, under certain conditions the action is the same in either variable.

\begin{theorem}[Theorem \ref{HeckeSym}]
Let $k$ be a positive even integer, and let $\Gamma \supseteq \Gamma_0(N)$ be a commensurable subgroup equipped with a multiplier $\nu$ that acts trivially on $\Gamma_0(N)$. Suppose that $S_k(\Gamma, \overline \nu)$ is trivial.
If $M$ is a positive integer coprime to $N$, then
 \[
\finfgen \nu {2-k} (z, \tau)\big|_{k,\tau}T_{M}= \finfgen \nu {2-k} (z, \tau)\big|_{2-k,z}T_{M}.
 \]
\end{theorem}
\begin{proof}
If $\im(\tau)>\im (z)$, then we may use \eqref{EqnHecke} to expand $\finfgen \nu {2-k} (z, \tau)\big|_{k,\tau}T_{M}
$ as a series in powers of $p$ to obtain

\begin{align*}
	\finfgen \nu {2-k} (z, \tau)\big|_{k,\tau}T_{M} &= \left(\sum_{m \in \tI { \nu} {2-k}} 	\finf {\nu} {2-k} {m} (z)p^m\right)\big|_{k,\tau}T_{M}\\
	 &= \sum_{m \in \tI {\nu} {2-k}} \sum_{d\,|\,\gcd(m,M)}
	d^{1-k/2}\left(\tfrac{M}{d}\right)^{k/2}	\finf {\nu} {2-k} {\tfrac{Mm}{d^2}} (z)p^m.\\
\end{align*}

If $\im(\tau)<\im (z)$, then we may expand $\finfgen \nu {2-k} (z, \tau)\big|_{k,\tau}T_{M}
$ as a series in powers of $q$ to obtain
\begin{align*}
	\finfgen \nu {2-k} (z, \tau)\big|_{k,\tau}T_{M} &= -\sum_{n \in \tJ {\overline\nu} {k}} \ginf {\overline\nu} {k} n (\tau)q^n\big|_{k,\tau}T_{M}.
\end{align*}
By hypothesis, $S_k(\Gamma,\overline \nu)$ is trivial, so $\widehat M_k^{(\infty)}(\Gamma,\overline\nu) \cap M_k(\Gamma, \overline{\nu})$ is spanned by at most a single Eisenstein series which has a constant at $\infty$ and vanishes at all other cusps. Thus a form in $\spcginf k \parent{\Gamma, \bar\nu}$ is determined uniquely by its principal part at infinity. The Hecke operator $T_m$ preserves the space $\spcginf k \parent{\Gamma, \bar\nu}$, so using \eqref{EqnHecke} with $d$ replaced with $M/d$ and considering the principal parts, we find for $n\geq0$ that
\begin{align*}
\ginf {\bar\nu} {k} n \big|_k T_M=\sum_{d\,|\,\gcd(n,M)} \left(\tfrac{M}{d}\right)^{k/2}d^{1-k/2} \ginf {\overline\nu} {k} {\tfrac{nM}{d^2}}.
\end{align*}
It follows that
\begin{align*}
	\finfgen \nu {2-k} (z, \tau)\big|_{k,\tau}T_{M} &= -\sum_{n \in \tJ {\overline\nu} {k}} \ginf {\overline\nu} {k} n (\tau)q^n\big|_{k,\tau}T_{M}\\
	&=\sum_{n \in \tJ {\overline\nu} {k}} \left(\sum_{d\,|\,\gcd(m,M)} \left(\tfrac{M}{d}\right)^{k/2} d^{1-k/2} \ginf {\overline\nu} {k} {\tfrac{nM}{d^2}}\right)q^n.\\
\end{align*}

Similar calculations allow us to expand $\finfgen \nu {2-k} (z, \tau)\big|_{2-k,z}T_{M}$ as a power series in either $p$ or $q$ depending on the relative sizes of $\im(z)$ and $\im(\tau).$ Here we use the fact that $S_{2-k}(\Gamma,\nu)$ is automatically trivial since $k\geq 2$.
By comparing expansions, we find that
\[
\finfgen \nu {2-k} (z, \tau)\big|_{k,\tau}T_{M}= \finfgen \nu {2-k} (z, \tau)\big|_{2-k,z}T_{M}.
 \]

\end{proof}
\section*{Declarations}
\subsection*{Data Availability}
Data sharing not applicable to this article as no datasets were generated or analysed during the current study.

\subsection*{Conflict of Interest Statement}
On behalf of all authors, the corresponding author states that there is no conflict of interest.

\bibliographystyle{amsplain}

\end{document}